\documentclass[11pt]{amsart} 

\usepackage{amssymb,amsmath,amscd,amsfonts,graphicx,latexsym,epsfig}
\usepackage{hyperref}

\usepackage[all]{xy}

\newtheorem{teo}{Theorem}[section]
\newtheorem{lem}[teo]{Lemma}

\newtheorem{cor}[teo]{Corollary}

\newtheorem{prop}[teo]{Proposition}

\newtheorem{ques}[teo]{Question}

\newtheorem{remark}[teo]{Remark}

\newcommand{\Aa}{{\mathcal A}}

\newcommand{\Dd}{{\mathcal D}}

\newcommand{\Ff}{{\mathcal F}}

\newcommand{\Pp}{{\mathcal P}}
\newcommand{\Qq}{{\mathcal Q}}

\newcommand{\Z}{{\mathbb Z}}

\newcommand{\X}{{\mathbb X}}

\newcommand{\D}{{\mathbb D}}
\newcommand{\N}{{\mathbb N}}
\newcommand{\F}{{\mathbb F}}

\DeclareMathOperator{\br}{br}

\def\Dim{\noindent\emph{Proof. }}
\def\cvd{\hfill$\Box$}

\newcommand{\compl}{\mathsf{C}}

\DeclareMathOperator{\lk}{lk}

\theoremstyle{definition}

\newtheorem{quest}[teo]{Question}

\begin{document}

\title[]{Alexander quandle lower bounds for link genera}

\author{R. Benedetti, R. Frigerio}

\email{benedett@dm.unipi.it, frigerio@dm.unipi.it}

\address{Dipartimento di Matematica \\
Universit\`a di Pisa \\
Largo B.~Pontecorvo 5 \\
56127 Pisa, Italy}




\begin{abstract} 
Every finite field $\F_q$, $q=p^n$, carries several {\it Alexander
    quandle} structures $\X=(\F_q,*)$. We denote by $\Qq_\Ff$
the family of these quandles, where $p$ and $n$ vary respectively among the odd primes and 
the positive integers.

For every $k$--components oriented link $L$, every
  partition $\Pp$ of $L$ into $h:= |\Pp|$ sublinks, and every
  labelling $\overline{z}\in\N^h$ of such a partition, the number of
  $\X$--colorings of any diagram of $(L,\overline{z})$ is a
  well--defined invariant of $(L,\Pp)$, of the form
  $q^{a_\X(L,\Pp,\overline{z})+1}$ for some natural number
  $a_\X(L,\Pp,\overline{z})$.  
Letting $\X$ and $\overline{z}$ vary
  respectively in $\Qq_\Ff$ and among the labellings of $\Pp$, we
  define the derived invariant $\Aa_\Qq(L,\Pp) := \sup
  \{a_\X(L,\Pp,\overline{z}) \}$.  

If $\Pp_M$ is such that
  $|\Pp_M|=k$, we show that $\Aa_\Qq(L,\Pp_M)\leq t(L)$, where $t(L)$ is
  the tunnel number of $L$, generalizing a result by Ishii. If $\Pp$
  is a {\it ``boundary partition''} of $L$ and $g(L,\Pp)$ denotes the
  infimum among the sums of the genera of a system of disjoint Seifert
  surfaces for the $L_j$'s, then we show that $\Aa_\Qq(L,\Pp)\leq
  2g(L,\Pp)+2k - |\Pp|-1 $. 
We point out further properties of
  $\Aa_\Qq(L,\Pp)$, mostly in the case of
  $\Aa_\Qq(L):=\Aa_\Qq(L,\Pp_m)$, $|\Pp_m|=1$. By elaborating on a
  suitable version of a result by Inoue, we show that when $L=K$ is a
  knot then $\Aa_\Qq(K)\leq \Aa(K)$, where $\Aa(K)$ is the breadth of
  the Alexander polynomial of $K$. However, for every $g\geq 1$ we
  exhibit examples of genus--$g$ knots having the same Alexander
  polynomial but different quandle invariants $\Aa_\Qq$. Moreover, in such examples $\Aa_\Qq$
  provides sharp lower bounds for the 
  genera of the knots. On the other hand, we show that $\Aa_\Qq(L)$ can give better
  lower bounds on the genus than $\Aa(L)$, when $L$ has $k\geq 2$ components.

We show that in order to compute $\Aa_\Qq(L)$ it is enough to
  consider only colorings with respect to the constant labelling $\overline{z}=1$.
In the case when $L=K$ is a knot, if either $\Aa_\Qq(K)=\Aa(K)$ or
  $\Aa_\Qq(K)$ provides a sharp lower bound for the knot genus, or
  if $\Aa_\Qq(K)=1$, then $\Aa_\Qq(K)$ can be realized by means of the
  proper subfamily of quandles $\{ \X=(\F_p,*)\}$, where $p$ varies among the odd primes.
\end{abstract}

\maketitle


\section{Introduction}\label{intro} 

A {\it quandle} $\X=(X,*)$ is a non-empty set $X$ with a
binary operation $\ast$ satisfying the following axioms: 
\smallskip
\begin{enumerate}
\item[(Q1)]
$a * a=a$ for every $a\in X$;
\item[(Q2)]
$(a * b) * c= (a * c) * (b * c)$ for every $a,b,c\in X$;
\item[(Q3)]
for every $b\in X$, 
the map $S_b\colon X\to X$ defined by $S_b(x)= x * b$ is a bijection.
\end{enumerate}

\smallskip

Every set $X$ admits the {\it trivial quandle} structure $\X(0)$ with
the operation defined by $a *^0 b=a$ for every $a,b\in X$.  Given a quandle
$\X=(X,*):= (X,*^1):=\X(1)$, for every integer $n>1$ one can define
another quandle $\X(n):=(X,*^n)$, where for every $a,b \in X$ one sets
$a*^nb=(a*^{n-1}b)*b$.
Every finite quandle $\X$ has a well defined
{\it type} $t_\X\geq 1$, such that $\X(n)=\X(m)$ if and only if $m=n$
mod $(t_\X)$.

\subsection{Quandle colorings}
Let $K\subset S^3$ be an oriented (smooth or PL) knot.  The
\emph{fundamental quandle} of $K$ was defined independently by
Joyce~\cite{joyce} and Matveev~\cite{matveev}. They also showed that
the fundamental quandle is a classifying invariant of knots.  If $\X$
is a {\it finite} quandle, then for every natural number $z\geq 0$ one
can define the invariant $c_{\X}(K,z)\in \N$ which counts the
representations of the fundamental quandle of $K$ in $\X(z)$.  It
turns out that $c_{\X}(K,z)$ can be computed as the number of suitably
defined $\X(z)$-{\it colorings} of any diagram $D$ of $K$. In order to
simplify the notation, we denote by $(K,z)$ a knot labelled by a
natural number $z$. Any label of $K$ obviously defines a label on
every diagram of $K$, and if $(D,z)$ is any diagram of $(K,z)$, then
we define a $\X$-coloring of $(D,z)$ to be a $\X(z)$-coloring of $D$.
Of course, if $\X$ has type $t_\X\geq 1$, then we may (and we will)
actually consider $\Z_{t_\X}$-valued (rather than $\N$-valued) labels,
where we understand that, for every $j\geq 2$, we identify
$\Z_j=\Z/j\Z$ with the set of canonical representatives
$\{0,\ldots,j-1\}$.
The definition of $c_\X(K,z)$ easily extends to the case of oriented
labelled links. In fact, let $L=K_1\cup\ldots\cup K_k$ be an oriented
link with $k$ components, and let $\Pp=(L_1,\ldots, L_h)$ be a
partition of $L$, where the $L_i$'s are disjoint sublinks of $L$ such
that $L=L_1\cup\ldots\cup L_h$. We denote by $| \Pp |=h$ the number of
links in the partition $\Pp$. A special r\^ole is played by the
maximal (resp.~minimal) partition $\Pp_M$ (resp.~$\Pp_m$) of $L$,
which can be characterized as the unique partition such that $| \Pp
|=k$ (resp.~$| \Pp |=1$), so that $L_i=K_i$ for $i=1,\ldots, k$
(resp.~$L_1=L$).  A ($\mathbb{N}$--valued) $\Pp$--cycle for $L$ is a
map $z\colon \{1,\ldots,h\} \to \Z$ that assigns the non--negative
integer $z_i=z(i)$ to every component of the sublink $L_i$ of $L$. In
what follows, we often denote such a cycle $(z_1,\ldots,z_h)$
simply by $\overline{z}$, and we denote by $\overline{0}$
(resp.~by $\overline{1}$) the cycle that assigns the integer $0$
(resp.~1) to every component of $L$.

If $D$ is a diagram of $L$, then any $\Pp$--cycle $\overline{z}$ for
$L$ descends to a $\Pp$--cycle $(D,\overline{z})$ for $D$. In
Section~\ref{quandle} we recall the definition of $\X$--coloring
of $(D, \overline{z})$. The total number of such colorings is denoted
by $c_\X (D, \Pp, \overline{z})$, and turns out to be independent of
the chosen diagram, thus defining an invariant
$c_\X(L,\Pp,\overline{z})$ of the partitioned and labelled link
$(L,\Pp,\overline{z})$.

\subsection{Alexander quandles}
In this paper we deal with a concrete family $\Qq_\Ff$ of finite
quandles, that we are now going to introduce.  Let us fix some
notation we will extensively use from now on.  For every odd prime
$p\geq 3$, we denote by $\Lambda$ (resp.~$\Lambda_m$) the ring
$\Z[t,t^{-1}]$ (resp.~$\Z_m[t,t^{-1}]$). Moreover,
$\pi_m\colon\Lambda\to\Lambda_m$ is the ring homomorphism induced by
the projection $\Z\to\Z_m$. For every $p(t)\in\Lambda$
(resp.~$p(t)\in\Lambda_m$) we define the \emph{breadth} $\br p(t)$ of
$p(t)$ as the difference between the highest and the lowest exponent
of the non--null monomials of $p(t)$. In particular, the breadth of
any constant polynomial (including the null polynomial) is equal to
$0$ (the reason why we set $\br 0=0$ will be clear soon).  If
$p(t),q(t)$ are elements of $\Lambda$ (resp.~of $\Lambda_m$), we write
$p(t)\doteq q(t)$ if $p(t)$ and $q(t)$ generate the same ideal of
$\Lambda$ (resp.~$\Lambda_m$), \emph{i.e.}~if and only if $p(t)=\pm
t^k q(t)$, $k\in\Z$ (resp.~$p(t)=at^k q(t)$, $a\in\Z_m^\ast$,
$k\in\Z$).

Recall that a finite {\it Alexander
quandle} is a pair $(M,*)$, where $M$ is
a finite $\Lambda_m$-module and the quandle operation is defined
(in terms of the module operations) by
$$a*b :=ta + (1-t)b \ . $$
We now define the family $\Qq_\Ff$ of finite Alexander quandles we are
interested in. Fix an odd prime $p$, let $h(t)$ be an
irreducible element of $\Lambda_p$ with positive breadth 
$\br h(t)=n\geq 1$, and let us define $\F(p,h(t))$ as the quotient
ring
$$
\F(p,h(t))=\Lambda_p/(h(t))\ .
$$
If $\hat{h}(t)\in\Z_p[t]\subseteq \Lambda_p$ is such that
$\hat{h}(t)\doteq h(t)$ and $h(0)\neq 0$, then it is readily seen the
the inclusion $\Z_p[t]\hookrightarrow \Lambda_p$ induces an
isomorphism $\Z_p[t]/(\hat{h}(t))\to \F(p,h(t))$. Since $\deg
\hat{h}(t)=\br h(t)=n$, it follows that $\F(p,h(t))$ is a finite field
of cardinality $q=p^n$.


We may therefore define the Alexander quandle $\X :=(\F(p,h(t)),*)$ by setting


$$ a*b:= \overline{t} a+(1-\overline{t})b\qquad {\rm for\ every}\ a,b\in 
\F(p,h(t))\ , $$
where $\overline{t}$
is the class of $t$ in $\F(p,h(t))$.
Once $q=p^n$ is fixed, there exists only one finite field $\F_q$ up to
field isomorphism.  However, even in the case
when $h_1(t)\in\Lambda_p$ and $h_2(t)\in\Lambda_p$
have the same breadth, it may happen that the quandles $(\F(p,h_1(t)),*)$ and
$(\F(p,h_2(t)),*)$ are not isomorphic
(see Remark~\ref{many-quandle}).

We now set 
$$\Qq_\Ff(m)= \{(\F(p,h(t)),*)\, |\,  
\ 1\leq \br h(t)\leq m \}  $$
and
$$\Qq_\Ff=\bigcup_{m\geq 1} \Qq_\Ff(m)\ .$$

\subsection{The invariant $\Aa_\Qq (L,\Pp)$}
Let us fix a quandle $\X=(F(p,h(t)),*)$.  
Let $D$ be a diagram of a link $L$, 
let $\Pp$ be a partition of $L$ and $\overline z$ be a $\Pp$--cycle
for $L$.
If $\X=(\F(p,h(t)),*)$, then
it turns out that the space
of the $\X$-colorings of $(D,\overline{z})$ is a
$\F(p,h(t))$-vector space of dimension $d_{\X}
(L,\Pp,\overline{z})\geq 1$.  Hence $c_{\X}(L,\Pp,
\overline{z})=q^{d_{\X}(L,\Pp,\overline{z})}$, so 
the whole information about $c_{\X}(L,\Pp,
\overline{z})$ is encoded by the integer
$ a_{\X}(L,\Pp,\overline{z}):=
d_{\X}(L,\Pp,\overline{z})-1 \geq 0$.
For instance, if $L=K$ is a knot,
then the $\Aa_\Qq$-{\it marked spectrum} of $K$, that is the set 
$\{a_{\X}(K,n)\, |\, \X\in \Qq_\Ff,\ n\in \Z_{t_\X}\}$, 
considered as a map defined on 
a subset of $\Qq_\Ff \times \N$, 
carries the whole
information provided by these quandle coloring invariants. In this paper we 
concentrate our attention on the derived invariant defined by
$$ \Aa_\Qq(L,\Pp):= \sup \{ a_\X(L,\Pp,\overline{z}) \}\ ,$$
where $\X$ varies in $\Qq_\Ff$ and $\overline{z}$ varies among the
$\Pp$--cycles of $L$. We  show 
in Lemma~\ref{unoriented}
that $\Aa_\Qq(L,\Pp_M)$ is an
invariant of the \emph{unoriented} link $L$.  On the contrary, for a
generic partition $\Pp$ the invariant $\Aa_\Qq(L,\Pp)$ can depend on
the orientations of the components of $L$.  For every partition $\Pp$
we have of course
$$
\Aa_\Qq(L,\Pp_m)\leq \Aa_\Qq (L,\Pp)\leq \Aa_\Qq(L,\Pp_M) \ .
$$

When $L=K$ is a knot, of course there is only one partition ($\Pp_m=\Pp_M$)
and we  simply write $\Aa_\Qq(K)$. Moreover, henceforth the 
invariant $\Aa_\Qq(L,\Pp_m)$ will be denoted simply by $\Aa_\Qq(L)$.

\subsection{A lower bound on the tunnel number of links}
Recall that the {\it tunnel number} $\ t(L)$ of a link $L\subset S^3$ is
the minimum number of properly embedded arcs in $S^3\setminus L$ to be
attached to $L$ in such a way that the regular neighbourhood of the
resulting connected spatial graph is an {\it unknotted} handlebody
(\emph{i.e.}~it is the regular neighbourhood also of a graph lying
on a $2$--dimensional sphere $S^2\subseteq S^3$). Of course, the tunnel
number is an invariant of \emph{unoriented} links.
\smallskip

The argument of
Proposition 6 in \cite{ishii1} (originally given for quandles of type
2) easily extends to our situation (see Proposition \ref{ishii})
and allows us to prove (in Subsection~\ref{tunnel:sub}) the following:

\begin{prop}\label{tunnel} For every link $L$ we have
$$ \Aa_\Qq(L,\Pp_M) \leq t(L) \ . $$
\end{prop}
In particular, $\Aa_\Qq(L,\Pp)$ is always finite.

\subsection{Lower bounds on genera of links}
We say that $\Pp=(L_1,\ldots,L_h)$ is a {\it boundary partition} of
$L=K_1\cup\ldots\cup K_k$ if there exists a system $(\Sigma_1,\ldots, \Sigma_h)$
of \emph{disjoint} connected oriented surfaces such that $\Sigma_i$ is a
Seifert surface of $L_i$ (i.e. $\partial \Sigma_i = L_i$ as oriented $1$--manifolds, where $\partial \Sigma_i$ inherits the orientation induced by $\Sigma_i$), 
for every
$i=1,\ldots,h$.  If $\Pp$ is a boundary partition of $L$, then we
define the \emph{genus} of $(L,\Pp)$ by
$$g(L,\Pp):= \min \left\{ \sum_{i=1}^h g(\Sigma_i)\right\}\ ,$$
where $(\Sigma_1,\ldots,\Sigma_h)$ varies among
such
systems of Seifert surfaces. If $\Pp$ is not a boundary partition, we set
$$g(L,\Pp)=+\infty \ . $$
 
Every link admits a connected Seifert surface, so $\Pp_m$ is
always a boundary partition, and the number $g(L):= g(L,\Pp_m)$ is
usually known as the \emph{genus} of $L$. On the other hand, $\Pp_M$
is a boundary partition if and only if $L$ is a boundary link. It is
immediate that $g(L,\Pp_M)$ is an invariant of the \emph{unoriented}
link $L$.
 
\medskip

The following result provides the fundamental estimate on link genera 
provided by quandle invariants, and
is proved in
Section~\ref{main-dim} (note that the statement below is non--trivial only when
$\Pp$ is a boundary partition):

\begin{teo}\label{generalbound} 
Let $(L,\Pp)$ be a $k$--component partitioned link, 
and let
$\overline{z}_1,\overline{z}_2$ be $\Pp$--cycles for $L$. Then we have:
$$ |a_\X(L,\Pp,\overline{z}_1)-a_\X(L,\Pp, \overline{z}_2)|\leq 2g(L,\Pp) + 
k - | \Pp | \ . $$
\end{teo}

Recall that $\overline{0}$ is the cycle that assigns the
integer $0$ to every component of $L$. In the
hypotheses of the previous Theorem, for every partition $\Pp$ we
have $a_\X (L,\Pp,\overline{0})=k-1$. Hence Theorem~\ref{generalbound}
immediately implies the following Corollary:

\begin{cor}\label{lowerbound}
If $(L,\Pp)$ is a $k$--component partitioned link, then:
$$ \Aa_\Qq(L,\Pp)\leq 2g(L,\Pp) + 2k - | \Pp | - 1\ . $$
In particular:
\begin{itemize}
\item 
If $\Pp_M$ is the maximal partition of $L$, then
$$
\Aa_\Qq (L,\Pp_M)\leq 2g(L,\Pp_M) +k-1\ .
$$
\item 
If $\Pp_m$ is the minimal partition of $L$, then
$$
\Aa_\Qq(L)=\Aa_\Qq (L,\Pp_m)\leq 2g(L) +2k-2\ .
$$
\item 
 For every knot $K$ we have
$$
\Aa_\Qq(K)\leq 2g(K)\ .
$$
\end{itemize}
\end{cor}

\begin{remark}\label{effectivebound}{\rm Let $L=K$ be a knot.  Clearly
    if $\Aa_\Qq(K)=2h$ is even, then $g(K)\geq h$; if
    $\Aa_\Qq(K)=2h-1$ is odd, then again $g(K)\geq h$. In particular,
    if $g(K)=g$, the bound on the genus provided by $\Aa_\Qq (K)$ is sharp if and only if $\Aa_\Qq(K)\geq
    2g-1$. The very same remark also applies in the general case of
    partitioned links.}
\end{remark}

\subsection{Alexander ideals and quandle coloring invariants of links}
\label{performance}
Once Theorem~\ref{generalbound} and Corollary~\ref{lowerbound} are established,
we will discuss a bit the performances of the $\Aa_\Qq(L,\Pp)$'s as
link invariants as well as lower bounds for the link genera. We will
mostly concentrate on the case $\Pp=\Pp_m$. 

\smallskip

The statement of Theorem~\ref{lowerbound} reminds the classical lower
bound (see \emph{e.g.}~\cite[Theorem 7.2.1]{crom})
$$ \Aa(L) \leq 2g(L)+k -1\ ,$$
where $\Aa(L)$ is the breadth of the Alexander polynomial
$$\Delta(L)(t):= {\rm det}\ \left(S(L) - t S(L)^T\right)\ ,$$
$S(L)$ being any Seifert matrix of $L$ 
(of course, the above estimate holds only if we agree that
the breadth of the null polynomial is equal to $0$).

Let us introduce some notations that will prove useful in describing
the relations between Alexander polynomial invariants and quandle
coloring invariants of links. We refer to~\cite{hill,hill2} for the
definitions and some basic results about Alexander ideals of links and
modules.  As usual, we denote by $p$ an odd prime number.  
If $K,K'$ are disjoint oriented knots in $S^3$, we denote by
$\lk (K,K')$ the usual \emph{linking number} of $K$ and $K'$. 
For every
oriented link $L=K_1\cup\ldots\cup K_k$, let $\widetilde{X}(L)$ the
\emph{total linking number} covering of the complement $\compl (L)$ of
$L$ in $S^3$, \emph{i.e.}~the covering associated to the kernel of the
homomorphism $\alpha\colon \pi_1 (\compl(L))\to\Z$, $\alpha
(\gamma)=\sum_{i=1}^k {\rm lk} (\gamma,K_i)$.  The covering
$\widetilde{X}(L)\to \compl(L)$ is infinite cyclic, so the homology
group $A(L)=H_1(\widetilde{X}(L);\Z)$
(resp.~$A^{(p)}(L)=H_1(\widetilde{X}(L);\Z_p)$) admits a natural
structure of $\Lambda$--module (resp.~$\Lambda_p$--module) such that
$t\in\Lambda$ (resp.~$t\in\Lambda_p$) acts on $A(L)$ (resp.~$A^{(p)}(L)$) 
as the map induced
by the covering translation corresponding to a loop $\gamma\in\pi_1(\compl(L))$
such that $\alpha(\gamma)=1$.
Let $E_i
(L)\subseteq \Lambda$ (resp.~$E_i^{(p)} (L)\subseteq \Lambda_p$) be the
$i$--th elementary ideal of $A(L)$ (resp.~$A^{(p)}(L)$). Since
$\Lambda$ is a U.F.D. (resp.~$\Lambda_p$ is a P.I.D.), for every
$i\geq 1$ it makes sense to define $\Delta_i (L)\in\Lambda$
(resp.~$\Delta_i^{(p)} (L)\in\Lambda$) as the generator of the
smallest principal ideal containing $E_{i-1}(L)$ (resp.~of the ideal
$E^{(p)}_{i-1}(L)$).  Then, $\Delta_i(L) (t)$
(resp.~$\Delta_i^{(p)}(L)(t)$) is well--defined only up to invertibles
in $\Lambda$ (resp.~$\Lambda_p$), \emph{i.e.}~up to multiplication by
$\pm t^k$, $k\in\Z$ (resp.~$at^k$, $a\in\Z_p^\ast$, $k\in\Z$).
Since 
$A(L)$ admits the square presentation matrix
$S(L)-tS(L)^T$ we have
$$
\Delta_1(L)(t)=\det \left(S(L)-tS(L)^T\right)=\Delta(L)(t)\ ,
$$
so $\Delta_1 (L)(t)$ coincides with the Alexander polynomial of $L$.
Some relations between $\Delta_i(L)(t)$ and $\Delta_i^{(p)}(L)(t)$ are
described in Corollary~\ref{princ:cor} (but see also
Remarks~\ref{inoue:rem} and~\ref{controes}).

Recall that $E^{(p)}_{i}(L)\subseteq E^{(p)}_{i+1}(L)$ for every
$i\in\N$, so either $\Delta_i^{(p)}(L)(t)=\Delta^{(p)}_{i+1}(L)(t)=0$,
or $\Delta^{(p)}_{i+1}(L)(t)$ divides $\Delta^{(p)}_i(L)(t)$ in
$\Lambda_p$.  Therefore, it makes sense to define the polynomial
$e^{(p)}_i(L) (t)\in\Lambda_p$ as follows:
$$
\begin{array}{lllc}
e^{(p)}_i(L)(t)&=&0\qquad & {\rm if}\ \Delta^{(p)}_i(L)(t)=
\Delta^{(p)}_{i+1}(L)(t)=0,\\ 
e^{(p)}_i(L)(t)&=&\frac{\Delta^{(p)}_i(L)(t)}{\Delta^{(p)}_{i+1}(L)(t)}\qquad 
& {\rm otherwise} \ .
\end{array}
$$
Also recall (see Lemma~\ref{reduce:lemma}) that there exists a minimum
$i_0\in\N$ such that $\Delta^{(p)}_i(L)(t)=\pm 1$ for $i\geq i_0$, whence
$e^{(p)}_i(L)(t)=\pm 1$ for every $i\geq i_0$.

In the very same way we can define the family of polynomials with integer
coefficients $\{e_i(L)(t)\}$ in $\Lambda$.

\medskip


In Section~\ref{AvsA} we prove the following result, which is strongly
related with the main result of~\cite{Inoue}, although there is a
subtlety in the statement that we will point out below.

\begin{teo}\label{inoue:teo}
Suppose $L$ is a link, and take a quandle
$\X=(\F(p,h(t)),\ast)\in\Qq_\Ff$. Then the space of $\X$--colorings
of $(L,\Pp_m,\overline{z})$ is in bijection with the module
$$
\F(p,h(t))\oplus \left(\bigoplus_{i=1}^\infty \Lambda_p 
\Big/\left({e}^{(p)}_i(L)(t^z), h(t)\right)\right) \ ,
$$
where $z=\overline{z}(\Pp_m)$ is the value assigned by $\overline{z}$
to every component of $L$, and
$(e_i^{(p)}(L)(t), h(t))\subseteq \Lambda_p$ is the ideal generated by 
${e}^{(p)}_i (t)$ and $h(t)$.
\end{teo}

Let us compare our result with Inoue's Theorem~\cite[Theorem 1]{Inoue}. 
We first observe that
in~\cite[Theorem 1]{Inoue} only the case
when $L=K$ is a knot and $\overline z = \overline 1$ 
is considered. Moreover, our proof of Theorem~\ref{inoue:teo}
does not make use of Fox differential calculus, and is therefore
quite different from Inoue's argument.
However, maybe the most interesting feature of the statement of
Theorem~\ref{inoue:teo} is that
%

\smallskip
\noindent
{\it The polynomial
  $e_i^{(p)}(L)(t)\in\Lambda_p$ is {\it not} the reduction mod {(p)}, say
  $\pi_p(e_i(L)(t))$, of $e_i(L)(t)$} 
\smallskip

\noindent as it could be suggested by the original statement of
~\cite[Theorem1]{Inoue}.  In fact, in Remark~\ref{inoue:rem} we 
show that the statement of Theorem~\ref{inoue:teo} does not hold
if the $e_i^{(p)}(L)(t)$'s are replaced by the
$\pi_p(e_i(L)(t))$'s.
In other words, in the following statement from the abstract of~\cite{Inoue}:

\smallskip
\noindent
{\it `` The number of all quandle homomorphisms of a knot quandle of a
  knot to an Alexander quandle is completely determined by Alexander
  polynomials of the knot''} 

\smallskip

\noindent the mentioned Alexander polynomials are not just the ones
relative to the usual Alexander $\Lambda$--module $A(L)$, but one has to
consider the polynomials associated to the whole family of
$\Lambda_p$--modules $A^{(p)}(L)$.

\smallskip

In the case of knots, building on Theorem~\ref{inoue:teo} we 
deduce (in Section~\ref{AvsA}) the following:

\begin{teo}\label{K=K} For every knot $K$ we have
$$
\Aa_\Qq(K)\leq \Aa (K) \ .
$$
Moreover
$\Aa_\Qq(K)= 0 $ if and only if $\Aa (K)=0$.
\end{teo}

In particular, as a bound on the genus of knots, the invariant $\Aa_\Qq(K)$ 
is dominated by $\Aa (K)$.
Moreover,
the following example shows that, when $L=K$ is a knot, the difference
between $\Aa (K)$ and $\Aa_\Qq (K)$ may become arbitrarily large.

For any pair $p,q$ of coprime integers, the torus knot $T_{p,q}$ has
tunnel number $t(T_{p,q})=1$ (and its unknotting tunnels have
been classified in \cite{BRZ}). Denoting by $\Delta_{p,q}(t)$ the
Alexander polynomial of $T_{p,q}$, it is well--known (see
\emph{e.g}~\cite[page 128]{burde}) that:
\begin{equation*}
\Delta_{p,q} (t)=\frac{(t^{pq} -1)(t-1)}{(t^p-1)(t^q-1)},\qquad
g(T_{p,q})=\frac{(p-1)(q-1)}{2}\ .
\end{equation*}
In particular, the bound on the genus of $T_{p,q}$ provided by the Alexander
polynomial is sharp, \emph{i.e.}~we have $\Aa (T_{p,q})=2g(T_{p,q})$.
As a consequence, 
we get the following:
\begin{prop}\label{via_crom} 
For every $n_0\in\mathbb{N}$ there exist an integer
$n\geq n_0$ and a knot $K$ such that
$$\Aa_\Qq(K)\leq  t(K)= 1 < n = 2g(K)=\Aa(K)\ . $$
\end{prop}
\smallskip

While being dominated by $\Aa (\cdot)$ in the case of knots, the
quandle invariant $\Aa_\Qq(\cdot)$ may provide a better lower bound on
the genus of $k$--component links, $k\geq 2$. Moreover,
$\Aa_\Qq(\cdot)$ can provide a sharp lower bound of the knot genus,
and can distinguish knots sharing both the genus and the Alexander
polynomial. 
More precisely, in Section~\ref{g-arbitrario} we prove the following 
Propositions:

\begin{prop}\label{nobounds}
For every $n\in\N$ there exists
a link $L$ such that $\Aa_\Qq(L)\geq n$ and $\Aa(L)=0$.
\end{prop}

\begin{prop}\label{betterinvariant}
Let us fix $g\geq 1$. Then, for every $r_1,r_2$ such that
$1\leq r_1\leq r_2 \leq 2r_1 \leq 2g$, there exist
  knots $K_1$ and $K_2$ such that the following conditions hold: 
$$g(K_1)=g(K_2)=g, \qquad  \Delta (K_1)=\Delta (K_2)
\ {\rm (whence}\  \Aa(K_1)=\Aa(K_2){\rm )}\ ,
$$
while $$\Aa_\Qq(K_1)=r_1, \qquad \Aa_\Qq(K_2)= r_2\ .$$ 
Moreover, we
  can require that both $\Aa_\Qq(K_1)$ and $\Aa_\Qq(K_2)$ are realized by
  means of some dihedral quandle with cycle $\overline{z}=1$. 
\end{prop}

\subsection{Further properties of the invariant $\Aa_\Qq$}\label{performance2}
Let $L=K$ be a knot, and let us look for proper subfamilies of
$\Qq_\Ff$ that carry the relevant information for computing
$\Aa_\Qq(K)$. 
In Lemma~\ref{z=1} we show that 
$\Aa_\Qq(K)$ is completely determined by the number of colorings
relative to the cycle $\overline{z}=1$: more precisely, we show that
for every knot $K$ there exists $\X\in\Qq_\Ff$ such that $\Aa_\Qq
(K)=a_\X (K,1)$.  In particular we can set
$$
\delta(K) := \inf \{n\in \N^*\  | \ \Aa_\Qq(K)= \sup \{a_\X(K,1)\ | \ \X \in
\Qq_\Ff(n)\}\} \ \in\ \N^*\ .
$$
$$
\theta(K):= \inf \{t_\X\ | \ \Aa_\Qq(K)= a_\X(K,1),\ 
\X \in \Qq_\Ff(\delta(K)) \} \ \in \N^*\ .
$$

\smallskip

In Subsection~\ref{subfamily} we prove the following: 


\begin{prop}\label{min}
Let $K$ be a knot.
\begin{enumerate}
\item If $\theta (K)>1$, then 
$\theta(K)\geq \delta(K)+1$.
\item If  $\Aa_\Qq(K)=1$, then $ \delta(K)= 1$.
\item If $\Aa(K)>0$, then $$\delta(K)\leq \frac{\Aa(K)}{\max
    \{2,\Aa_\Qq(K)\} }\ .$$
\item
Suppose that $\Aa_\Qq(K)=\Aa(K)$ or $\Aa_\Qq(K)=\Aa(K)-1$. Then $\delta(K)=1$.
Moreover, there
  exist an odd prime $p$ and an element
$a\in\Z_p^\ast$ such that $(t-a)^{\Aa_\Qq(K)}$ divides $\Delta_1^{(p)}(K)(t)$ in
  $\Lambda_p$.
\item 
If $\Aa_\Qq(K)=\Aa(K)$, then $\delta(K)=1$ and
there exist an odd prime $p$ and an element $a\in\Z_p^\ast$ such that
$\Delta_1^{(p)}(K)(t)\doteq (t-a)^{\Aa(K)}$ in $\Lambda_p$.
\end{enumerate}
\end{prop}

In Section \ref{g=1}, Corollary \ref{g=1completo}, we check
directly that if $g(K)=1$ then either $\Aa_\Qq(K)=0$ or $\Aa_\Qq(K)\in
\{ 1,2 \}$, and in the last case 
we have
$$(\delta(K),\theta(K))=(1,2) \ . $$

\smallskip

\begin{quest}\label{delta:quest}
{\rm
Let $n\in\N$ be fixed. Does a knot $K$ exist such that 
$\delta (K)\geq n$? 
(See Remark~\ref{delta:unbounded} for a brief discussion about this
issue).
} 
\end{quest}

\begin{ques}{\rm 
Is $\theta(K)$ bounded from above by an explicit function of $g(K)$ (or
$\Aa(K)$, or $\delta(K)$)?
}
\end{ques}


\section{Quandle invariants}\label{quandle} 
We briefly recall a few details about the definition of quandle
invariants of links and about our favourite family $\Qq_\Ff$ of finite
Alexander quandles.

\medskip


Let $\X=(X,*)$ be any finite quandle, $|X|=m$. For every $b\in X$, the
permutation of $X$ defined by $S_b\colon a\mapsto a*b$ has order
$o(b)$ that divides $m! \, $. If we denote by $t_\X$ the l.c.m.~of
these orders, then for every $a,b\in X$ we have $S_b^{t_\X}(a)=
a*^{t_\X}b=a$, that is $*^{t_\X}=*^0$, and it is readily seen that
$t_\X$ is in fact the type of $\X$, as defined in
the Introduction.

\subsection{Basic properties of finite Alexander quandles}
Let us now turn to our favourite Alexander quandles
$\X=(\F(p,h(t)),*)$ where $h(t)$ is an irreducible polynomial of
breadth $n\geq 1$ in $\Lambda_p$.
Hence $\F(p,h(t))$ is a finite field with $q=p^n$ elements.


For every $m\geq 0$ set $p_{m+1}(t)=\sum_{j=0}^{m}t^j\in \Z[t]$ and
$H_m (t)=1-t^m\in\Z[t]$, in such a way that $H_m (t)= (1-t)p_m(t)$ for
every $m\geq 1$ (when this does not arise ambiguities, we 
consider $p_m (t)$ (resp.~$H_m (t)$) also as elements of $\Z_p[t]$,
$\Lambda$ and $\Lambda_p$).  Also recall that $\overline{t}$ denotes
the class of $t$ in $\F(p,h(t))$.  An easy inductive argument shows
that for every $a,b \in \F(p,h(t))$ and every $m\geq 0$ we have
$$ a*^mb = \overline{t}^ma + H_m(\overline{t})b \ .$$


\begin{lem}\label{type} Let $\X=(\F(p,h(t)),*)$ be a
finite Alexander quandle as above, 
let $n=\br h(t)$ and set $q=p^n$. 
Then:
\begin{enumerate}
\item
$\X$ is trivial if and only if $h(t)\doteq t-1$.
\item If $\X$ is non--trivial, then $*^m=*^0$ (\emph{i.e.} $m$ is a
  multiple of $t_\X$) if and only if $h(t)$ divides $p_m(t)$ in
  $\Lambda_p$.
\item
$H_m (\overline{t})=0$ in $(\F_p,h(t))$
if and only if $m$ is a multiple of $t_\X$.
\item
Suppose that $\X$ is non--trivial. Then
$t_\X\geq n+1$.
Moreover,
$t_\X=n+1$ if and only $p_{n+1} (t)$ is irreducible in $\Z_p[t]$ and
$h(t)\doteq p_{n+1}(t)$ in $\Lambda_p$.  If this is the case, then $p_{n+1}(t)$ is
irreducible in $\Z[t]$, and $n+1$ is prime.
\end{enumerate}
\end{lem}
\begin{proof}
  (1) If $h(t)\doteq t-1$, then $\overline{t}=1$ in $\F(p,h(t))$, so
  $\overline{t} a+(1-\overline{t})b=a$ for every $a,b\in \F(p,h(t))$.
  On the other hand, if $\X$ is trivial, then
  $(1-\overline{t})(b-a)=0$ for every $a,b\in \F(p,h(t))$, so
  $1-\overline{t}=0$. This implies that $t-1$ divides $h(t)$ in
  $\Lambda_p$, so $h(t)\doteq t-1$ by irreducibility of $h(t)$.

(2) We have $ a*^mb= \overline{t}^ma +
(1-\overline{t})p_m(\overline{t})b= a$ if and only if
$(\overline{t}-1)p_m(\overline{t})(a-b)=0$. By point~(1)
this equality holds for every
$a,b$ if and only if $p_m(\overline{t})=0$ in $(\F_p,h(t))$,
\emph{i.e.}~if and only if $h(t)$ divides $p_m(t)$ in $\Lambda_p$.

(3) By point~(1), $\X$ has type $1$ (\emph{i.e.}~it is trivial) if and
only if $H_1(\overline{t})=0$, so we may suppose that $\X$ is
non--trivial. In this case, since $H_m
(\overline{t})=(1-\overline{t})p_m(\overline{t})$ and
$1-\overline{t}\neq 0$ in $(\F_p,h(t))$, point~(3) is an immediate
consequence of~(2).

(4) 
By point~(2) the polynomial $h(t)$
divides $p_{t_\X} (t)$ in $\Lambda_p$, so $n=\br h(t)\leq \br
p_{t_\X}= t_\X -1$, whence the first statement.
Moreover, $t_\X=n+1$ if and only if $h(t)$ divides $p_{n+1}(t)$ in $\Lambda_p$.  
Since $\br h(t)=\br p_{n+1}(t)$, this condition holds if
and only if $p_{n+1}(t)\doteq h(t)$, and this implies that
$p_{n+1}(t)$ is irreducible in $\Lambda_p$, whence in $\Z_p[t]$.
Being monic, if $p_{n+1}(t)$ is irreducible in $\Z_p[t]$, then it is
irreducible also in $\Z[t]$, and this implies in turn that $n+1$ is
prime.
\end{proof}

\medskip

The simplest non--trivial quandles in our family $\Qq_\Ff$ are the
{\it dihedral quandles} $\Dd_p=(\F(p,1+t),\ast)$. In this
case the quandle operation takes the form $a*b=2b-a$, in terms of the field
operations of $\F(p,h(t))=\Z_p$. Dihedral
quandles are \emph{involutory}, \emph{i.e.}~their type is equal to 2.


\begin{remark}\label{many-quandle}{\rm If $q=p^n$, the finite
field $\F_q$, which is unique up to isomorphism, supports in general 
non--isomorphic quandle structures.
This phenomenon shows up already when $n=1$, \emph{i.e.}~when
considering Alexander quandles in $\Qq_\Ff(1)$. For
    every odd prime $p$ and every $a\in\Z_p^*$, 
    let $h_a(t)= a+t$, and let $\X_{p,a}=(\F(p,h_a(t)),\ast)$ 
be the corresponding 
Alexander quandle. We have seen in Lemma~\ref{type}--(1)
that $\X_{p,a}$ is trivial if and only if $a=p-1$. On the other hand, if $a=1$
then $\X_{p,a}$ is a dihedral quandle, and its type is equal to 2. 
By Lemma~\ref{type}--(4), if $a\notin \{1,p-1\}$ then $t_{\X_{p,a}}>2$, 
so the quandles
$\X_{p,a}$, $\X_{p,1}$ and $\X_{p,p-1}$ are pairwise non--isomorphic. For example,
Lemma~\ref{type}--(2) implies that $t_{\X_{p,a}}=3$ if and only if $a\neq 0,-1$ 
and
$-a$ is a root of
$t^2+t+1$, \emph{i.e.}~if and only $p\neq 3$ and the equation 
$a^2-a+1=0$ has a root
in $\Z_p$ (such a root is necessarily distinct from $0,-1$). 
The discriminant of this quadratic equation
is equal to $-3$, so
we can conclude that $t_{\X_{p,a}}=3$ if and only
if $p\neq 3$, the element $p-3$ admits a square root $c$ in $\Z_p$,
and  $a=(1\pm c)(k+1)$. 

Also observe that, if $p>n(n-1)+2$, then there exists 
$a\in\Z_p\setminus\{0,-1\}$
such that $-a\in\Z_p$ is not a root
of $p_i(t)\in\Z_p[t]$ for every $i=1,\ldots,n$.  
By Lemma~\ref{type}--(2), this implies that the type of $\X_{p,a}$
exceeds $n$, and this shows that $\Qq_\Ff (1)$ contains quandles
of arbitrarily large type.

Here is another construction of non--isomorphic quandles supported by
the same finite field $\F_q$. Assume for example that both $1+t^m$ and
$p_{m+1}(t)$ are irreducible in $\Z_p[t]$. By Lemma~\ref{type}--(4),
the type of $(\F(p,p_{m+1}(t)),*)$ is equal to $m+1$. On the other
hand, since $(1+t^m)p_{m+1}(t)=p_{2m}(t)$, points~(4) and (2) of
Lemma~\ref{type} imply respectively that the type of $(\F(p,1+t^m),*)$
is bigger than $m$ and divides $2m$, and is therefore equal to
$2m$. An example of this kind is
obtained by taking $m=2$ and $p=11$, so that we have two
non--isomorphic quandle structures (of type $3$ and $4$ respectively)
on $\F_q$, where $q=11^2$.}
\end{remark}

\subsection{Quandle colorings of links}
Let $L=L_1\cup\ldots\cup L_h=K_1\cup\ldots\cup K_k$ be an oriented
partitioned link with $k$ components, where $\Pp=(L_1,\ldots,L_h)$ is
a partition of $L$ into sublinks, and let $D$ be any diagram of $L$.  A ($\Z_{t_\X}$--valued) $\Pp$--cycle on $L$ is a map
$\overline{z}\colon\{1,\dots , h\}\to \Z_{t_\X}$, where $\overline{z}(i)$ labels every
component of the sublink $L_i$.  Such a cycle naturally descends to
$D$. An {\it arc} of $D$ is any embedded open interval in $D$ whose
endpoints are undercrossing.  An $\X$-coloring of
$(D,\Pp,\overline{z})$ assigns to each arc of $D$ a ``color''
belonging to $\X$ in such a way that at every crossing we see the
local configuration shown in Figure \ref{quandle2}.  Here $a,b\in\X$
are colors, and $z$ refers to the value assigned by $\overline{z}$ to
the sublink that contains the overcrossing arc.

\begin{figure}[htbp]
\begin{center}
 \includegraphics[height=4cm]{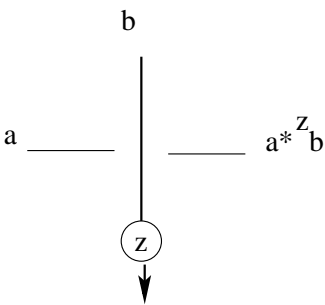}
\caption{\label{quandle2} The local configuration of a quandle coloring.}
\end{center}
\end{figure} 

\begin{remark}
{\rm
The case when $\X$ is a dihedral quandle is particularly simple to handle
because in this case orientations become immaterial from the very
beginning, in the sense that the rule of Figure \ref{quandle2} is 
well--defined even if one forgets  the orientation of the overcrossing arc.
}
\end{remark}

The following Proposition shows that 
$$c_\X(L,\Pp,\overline{z}):=c_\X(D,\Pp,\overline{z})$$ is a well defined
invariant of $(L,\Pp, \overline{z})$ (up to isotopy of oriented,
partitioned and labelled links), where $c_\X(D,\Pp,\overline{z})$ is
the number of $\X$-colorings of $(D,\Pp,\overline{z})$.

\begin{prop}\label{ishii:gen}
Let $(L,\Pp,\overline{z})$ be a partitioned link endowed with a fixed
$\Z_{t_\X}$--cycle, and let $D,D'$ be diagrams of $L$. Then we have
$$
c_\X(D,\Pp,\overline{z})=c_\X(D',\Pp,\overline{z})\ .
$$
\end{prop}
\begin{proof} 
  Let us briefly describe how our statement can be deduced from the
  results proved in~\cite{ishii1, ishii2} (in~\cite{ishii1} only the
  case of involutory quandles is considered, but such a restriction is
  overcome in~\cite{ishii2}). In order to check that $c_\X
  (D,\Pp,\overline{z})$ is independent of $D$ it is sufficient to
  prove the statement in the case when $D$ and $D'$ are related to
  each other by a classical Reidemeister move on oriented link
  diagrams. In the cited papers the authors consider indeed a more
  general situation, where $D$ and $D'$ are trivalent spatial graphs,
  and $D'$ is obtained from $D$ either via a Reidemeister move, or via
  a Whitehead's move (by the way, this ensures that $D$ and $D'$ have
  ambient--isotopic regular neighbourhoods in $S^3$ -- see also the
  discussion in Subsection \ref{tunnel:sub} below). In our case we
  have to deal only with the usual Reidemeister moves. Moreover, every
  $\Z_{t_\X}$--cycle on $D$ canonically defines a $\Z_{t_\X}$--cycle on $D'$, so
  the arguments in~\cite{ishii1, ishii2} prove the claimed result.
  \end{proof}

Let $\X \in \Qq_\Ff$ be a quandle of type $k$ supported by the field
$\F_q$.  It is clear that the $\X$-colorings of a diagram
$(D,\Pp,\overline{z})$ as above correspond to the solutions of a
linear system over $\F_q$. Therefore, the space of such colorings
(which contains all the constant colorings) is a $\F_q$-vector space
of dimension $d_\X(D,\Pp,\overline{z})\geq 1$, so
the whole information about $c_\X(D,\Pp,\overline{z})$
is encoded by the natural number
$$a_\X(D,\Pp,\overline{z}):= d_\X(D,\Pp,\overline{z})-1\ .$$
By Proposition~\ref{ishii:gen}, this number is a well defined isotopy
invariant of oriented and $\Z_{t_\X}$-labelled partitioned links.
As a consequence, the following polynomial, that collects all such
``monomial'' invariants, is an invariant of oriented partitioned
links:
$$ \Phi_\X(L,\Pp)(t):= \sum_{\overline{z}} t^{a_\X(L,\Pp,\overline{z})} \in \N[t] \ . $$
Also observe that by the very definitions we have
$$
\deg \Phi_\X(L,\Pp)(t)=\sup_{\overline{z}} a_\X(L,\Pp,\overline{z})\ ,
$$
whence 
\begin{equation}\label{tunnel1}
\Aa_\Qq(L,\Pp)=\sup_{\X\in\Qq_\Ff} \deg \Phi_\X(L,\Pp)(t)\ .
\end{equation}

\begin{lem}\label{unoriented}
  Let $L$ be an oriented link, and let $\Pp_M$ be its maximal
  partition.  Then the polynomial $ \Phi_\X(L,\Pp_M)(t)$ is an
  invariant of $L$ as an \emph{unoriented} link. As a consequence, 
$\Aa_\Qq(L,\Pp_M)$ is an
  invariant of $L$ as an \emph{unoriented} link.
\end{lem}
\begin{proof}
  Let $\Pp_M=(K_1,\ldots,K_h)$ be the maximal partition of $L$, and
  for every $\epsilon\colon \{1,\dots, h\} \to \{\pm 1 \}$ let us
  denote by $\epsilon L$ the link $\epsilon(1) K_1 \cup \dots \cup
  \epsilon(h)K_h$, where as usual the symbols $K$ and $-K$ denote
  knots having the same support and opposite orientations. We also
  define the cycle $\epsilon \overline{z}$ by setting $(\epsilon
  \overline{z}) (j) = \epsilon(j)\overline{z}(j)$. It is not hard to
  verify that for every cycle $\overline{z}$ and every $\epsilon$ we
  have
$$ a_\X(\epsilon L,\Pp_M,\overline{z})= 
a_\X(L,\Pp_M,\epsilon \overline{z}) \ . $$
\smallskip

We now say that two cycles $\overline{z}$ and $\overline{z}'$ are {\it
  equivalent} if and only if there exists $\epsilon$ such that
$\overline{z}'=\epsilon \overline{z}$, and we denote by
$[\overline{z}]$ the equivalence class of $\overline{z}$. The previous
discussion shows that the polynomials
$$\Phi_\X(L,\Pp,[\overline{z}])(t):= \sum_{\overline{z}'\in [\overline{z}]} 
t^{a_\X(L,\Pp,\overline{z}')}$$ 
do not depend on the orientation of the components of $L$.
The conclusion now follows from the obvious equality
$$  \Phi_\X(L,\Pp_M)(t):= \sum_{[\overline{z}]} 
\Phi_\X(L,\Pp_M,[\overline{z}])(t)\ . $$
\end{proof}

\subsection{Quandle invariants and tunnel number}\label{tunnel:sub}
For every finite quandle $\X$, the number $c_\X(L,\Pp_m,\overline{1})$
(that is the number of colorings associated to the cycle assigning the
value $1$ to every component of $L$) is in a sense the most widely
considered quandle coloring invariant of classical links.  The {\it
  multiset} of invariants obtained by varying the $\Z_{t_\X}$-cycles
(when $\X \in \Qq_\Ff$, such a multiset is encoded by the polynomial
$\Phi_\X(L,\Pp_M)(t)$) has been introduced in \cite{ishii1, ishii2} in
order to extend quandle coloring invariants to {\it spatial
  graphs} and even to {\it spatial handlebodies}.  It turns out that
this approach is useful also in the case of links. An interesting
application of these extended invariants is given in~\cite[Proposition
6]{ishii1}, where only quandles of type 2 are considered.  However,
Ishii's argument applies verbatim to our (more general) case, thus
giving the following:

\begin{prop}\label{ishii} 
For every $\X\in \Qq_\Ff$ and every (unoriented) link $L$ we have
$$t(L) \geq \Aa_\Qq (L,\Pp_M) \ . $$
\end{prop}
\Dim We sketch the proof for the sake of completeness. 
By equality~\eqref{tunnel1}, it is sufficient to show that
$t(L)\geq \deg \Phi_\X (L,\Pp_M)(t)$ for every quandle
$\X\in\Qq_\Ff$.
Set $m=t(L)$.
Then there is a sequence $L:=G_0 \subset G_1 \subset G_2 \subset \dots
\subset G_m$, where $G_j$ is a spatial graph with trivalent vertices
obtained by attaching an arc to $G_{j-1}$, and $G_m$ is a spine of an
unknotted handlebody. According to~\cite{ishii1}, for every quandle
$\X$ the $\X$-colorings of any diagram of a trivalent graph like $G_j$
verify (in addition to the rule already described in Figure
\ref{quandle2}) the further vertex condition described on the left of
Figure \ref{quandle3} (here $a$ refers to a color). With such a
definition of coloring, the number of $\X$--colorings of the diagram
of a spatial graph does depend only on the isotopy class of a regular
neighbourhood of the graph, which is a spatial handlebody (the proof
of Theorem~5 in~\cite{ishii1} does not really makes use of condition
(K2') stated there, that is equivalent to asking that the considered
quandle has type 2).


\begin{figure}[htbp]
\begin{center}
 \includegraphics[height=5cm]{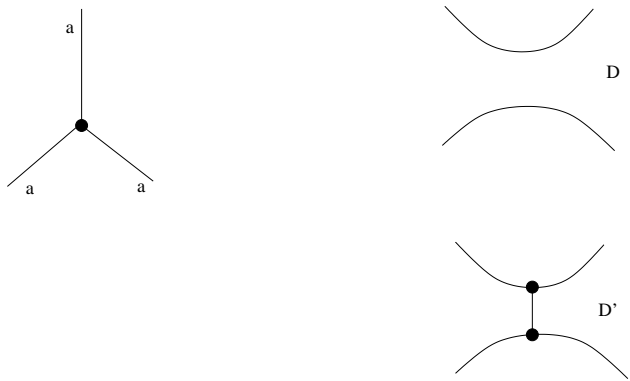}
\caption{\label{quandle3} Quandle colorings at vertices of trivalent graphs.}
\end{center}
\end{figure} 

We can assume that $G_{j-1}$ and $G_j$ admit respectively diagrams $D$
and $D'$ that differ from each other only by the local configurations
shown on the right of Figure \ref{quandle3}. Every cycle on $G_{j-1}$
extends to a cycle on $G_j$ that assigns the value $0$ to the added
arc. Then it is easy to show that
$$ \deg \Phi_\X(G_j)(t) \geq  \deg \Phi_\X(G_{j-1})(t) -1\ .$$ 
Moreover, since a regular neighbourhood of $G_m$ is an unknotted
handlebody, we have
$$ \deg \Phi_\X(G_m)(t)=0\ ,$$
hence 
$$0\geq  \deg \Phi_\X(L)(t) -t(L) \ . $$
\cvd

Proposition \ref{tunnel} is now an easy consequence of
Proposition~\ref{ishii}.  In \cite{RR} we have used 
Ishii's quandle coloring invariants of graphs
(only exploiting the dihedral case)
in order to detect different
{\it level of knottings} of spatial handlebodies.

\section{Ribbon tangles}\label{ribbon}
Let us now fix a quandle $\X \in \Qq_\Ff$ of type $t_\X$.
The following simple Lemma (it is a straightforward computation) 
plays a crucial r\^ole in the proof of our main results. Consider the local
configurations of Figure \ref{doppio1}. Here $a,b,c,b_1,b_2$ are
colors belonging to some $\X$-coloring, where we understand that $z\in
\{0,\dots, t_\X-1\}$ is the same value of the cycle on {\it both}
the overcrossing strands.

\begin{figure}[htbp]
\begin{center}
 \includegraphics[height=4cm]{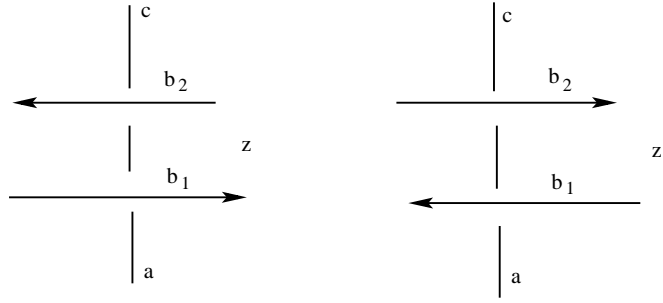}
\caption{\label{doppio1} 
The quandle colorings described in Lemma~\ref{const-diff}.}
\end{center}
\end{figure}  

\begin{lem}\label{const-diff} For the diagram on the left of Figure
  \ref{doppio1} we have:
$$c= a+ \overline{t}^{-z}H_z(\overline{t})(b_1-b_2) \ .$$
For the diagram on the right we have:
$$c= a+H_z(\overline{t})(b_2-b_1)\ . $$
\end{lem}

Let us consider a decorated tangle diagram $T$ as
suggested in Figure \ref{ribbontangle1}. 

It is understood that the circular box contains $h$ oriented strings,
each of which has an ``input'' and an ``output'' endpoint.
Moreover, the $j$--th string is decorated with a sign $s_j\in
\{\pm 1\}$, and its endpoints 
are endowed 
with an input color $r_j$ and
an output color $f_j$.

\begin{figure}[htbp]
\begin{center}
 \includegraphics[height=7cm]{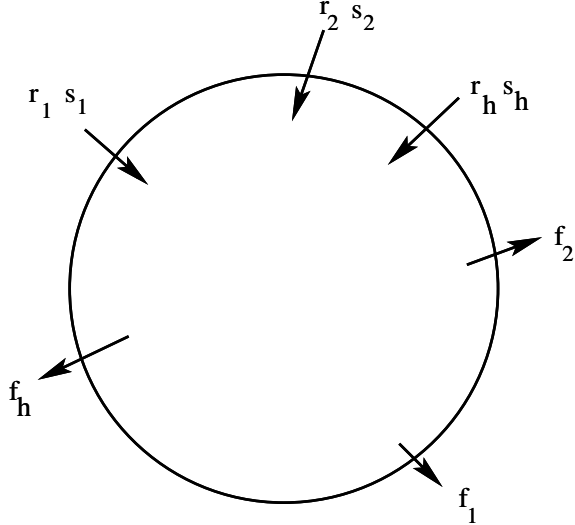}
\caption{\label{ribbontangle1} A string tangle $T$.}
\end{center}
\end{figure}  

We use such a string tangle to encode an associated {\it ribbon
  tangle} $R(T)$ with oriented {\it ribbon boundary tangle} $D(T)$, by
applying the doubling rules suggested in Figure \ref{ribbontangle2},
where the left (right) side refers to the string sign $s = 1$
($s=-1$). Every ribbon component has two oriented boundary components,
that are two copies of the corresponding string of $T$ with opposite
orientations.  These boundary components are also ordered by taking
first the component which shares the same orientation as the
corresponding string of $T$.

\begin{figure}[htbp]
\begin{center}
 \includegraphics[height=4cm]{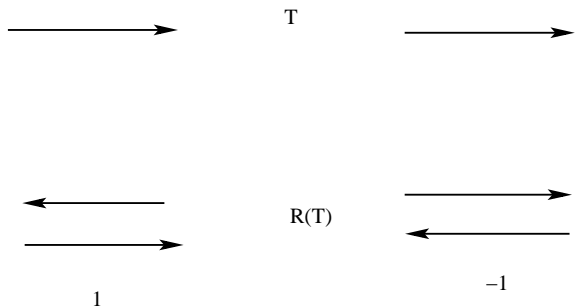}
\caption{\label{ribbontangle2} From a string tangle to a ribbon tangle.}
\end{center}
\end{figure}  

If $\overline{z}\colon \{1,\dots , h\}\to \Z_{t_\X}$ is any cycle defined
on the strings of $T$, we define the associated {\it ribbon boundary
  cycle} $\hat{z}$ on $D(T)$ by assigning the same value
$\overline{z}(j)$ to both boundary components of the ribbon associated
to the $j$--th string of $T$. In this way we have obtained a {\it
  $\Z_{t_\X}$-labelled ribbon boundary tangle} $(D(T),\hat{z})$. Arcs of
$T$ and of $D(T)$ are defined as usual, provided now that also the
endpoints of the strings of $T$ and $D(T)$ have to be considered as
endpoints of arcs of $T$ and $D(T)$.  \smallskip

The notion of $\X$-coloring extends obviously to any $\Z_{t_\X}$--labelled
ribbon boundary tangle $(D(T),\hat{z})$.  For every such a coloring,
along every arc of $T$ we see a couple of ordered arcs of $D(T)$
carrying an ordered couple of colors, say $(a,b)$.  The following
result is an immediate consequence of Lemma \ref{const-diff}.

\begin{lem}\label{const-diff2} For every $\X$-coloring of $(D(T),\hat{z})$,
  the {\rm color difference} $d= b-a$ is constant along every
  string of $T$.
\end{lem}

Then every such $\X$-coloring can be described as follows.
At the input point of the $j$--th string of $T$ we have an ordered couple of
colors $(a_j,\ a_j+d_j)$. Along every arc $\alpha$ of $T$ belonging to 
the $j$--th string, we have a couple of colors of the form 
$(a_j+r_\alpha,\  a_j + d_j + r_\alpha)$.

\medskip

For obvious reasons, we say that the $r_\alpha$'s define an $\X_{\rm
  diff}$-{\it coloring} of the arcs of $T$, which vanishes at the
input points of the strings of $T$ (observe that the definition of
difference between colors relies on the fact that $X$ is a module, and
is not related to the quandle operation of $\X$).  We now deduce from
Lemma~\ref{const-diff} the rule governing these $\X_{\rm
  diff}$-colorings at crossings.  We refer to Figure
\ref{diff-rule3}. Here $r_I,r_F$ are $\X_{\rm diff}$-colors, $d$, $z$,
$s$ are respectively the constant $\X$-color difference, the value of
the cycle and the sign of the string that contains the overcrossing
strand, and $\epsilon=\pm 1$ is the usual sign of the crossing.

\begin{figure}[htbp]
\begin{center}
\includegraphics[height=3cm]{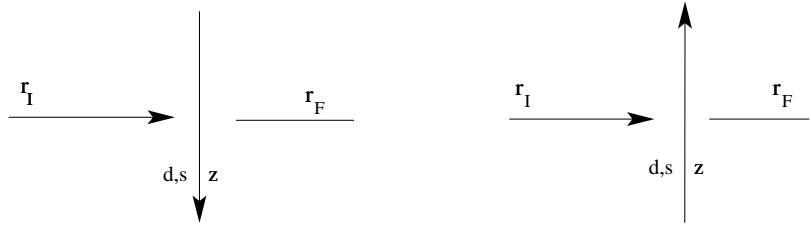}
 \caption{\label{diff-rule3} The behaviour of $\X_{\rm diff}$ at crossings. On the top,
   a positive crossing.  On the bottom, a negative crossing.}
\end{center}
\end{figure}  
\smallskip

With notation as in Figure~\ref{diff-rule3}, 
Lemma~\ref{const-diff} readily implies that
\begin{equation}\label{diff-eq}
r_F = r_I - \epsilon H_z(\overline{t})\overline{t}^{\sigma(s)z}d \ . 
\end{equation}

As a consequence, every $\X_{\rm diff}$-coloring of $(T,\overline{z})$
(in particular the corresponding set of output colors $\{f_j\}$) is
completely determined by the input data $\{d_j\}$, and every
$\X$-coloring of $(D(T),\hat{z})$ is completely determined by the
input data $\{(a_j,d_j)\}$.
In fact, every $\X_{\rm diff}$-coloring of $(T,\overline{z})$
can be constructed as follows: we run along every string of $T$ from
its input point to its output point, and at every undercrossing we add
to the local input value $r_I$ a suitable term according to
equation~\eqref{diff-eq}. 


Given an ordered couple $(i,j)$ of string indices, let $n_{i,j}^+$
(resp.~$n_{i,j}^-$) be the number of times the $i$--th string passes under
the $j$--th string at a positive (resp.~negative) crossing, and let us set
$$M_{i,j}= n_{i,j}^+ - n_{i,j}^- \ . $$
The following Proposition summarizes the discussion carried out in
this Section.

\begin{prop}\label{diff-out} 
  Let $T$ be a decorated tangle with associated ribbon boundary tangle
  $D(T)$.  Then, every $\X_{\rm diff}$-coloring of $(T,\overline{z})$
  (in particular the corresponding set of output colors $\{f_j\}$) is
  completely determined by the input data $\{d_j\}$, and every
  $\X$-coloring of $(D(T),\hat{z})$ is completely determined by the
  input data $\{(a_j,d_j)\}$. In particular, the $f_j$'s can be
  computed in terms of the $d_i$'s by means of the formula
$$f_i = 
-\sum_{j=1}^h
M_{i,j}\overline{t}^{\sigma(s(j))\overline{z}(j)}H_{\overline{z}(j)}(\overline{t})d_j
\ . $$
\end{prop}

\section{Seifert surfaces and special diagrams}\label{specialdiag}
Let us consider a compact oriented surface $\Sigma_{g,s}$ of genus $g$
having $s\geq 1$ boundary components. Clearly $g+s\geq 1$, and $g+s=1$
if and only if $g=0$ and $s=1$, \emph{i.e.}~if $\Sigma_{g,s}$ is a disk. 
Let us assume that $g+s>1$.  It is well
known that $\Sigma_{g,s}$ is homeomorphic to the model shown in
Figure~\ref{surface}, where the case $g=2$, $s=3$ is considered. The
picture stresses also the fact that $\Sigma_{g,s}$ is the regular
neighbourhood of a 1--dimensional trivalent graph $P_{g,s}$, which is
therefore a \emph{spine} of $\Sigma_{g,s}$.

\begin{figure}[htbp]
\begin{center}
 \includegraphics[height=5.5cm]{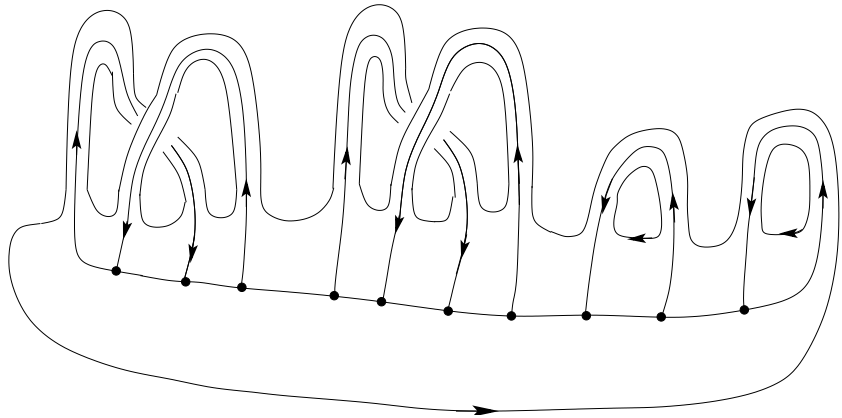}
\caption{\label{surface} The surface $\Sigma_{2,3}$ and the spine $P_{2,3}$.}
\end{center}
\end{figure}  

Let now $L$ be an oriented link endowed with a Seifert surface $\Sigma$ of
genus $g$, and let $s$ be the number of components of $L$. Assume
first that $g+s>1$ . Then the pair $(\Sigma,L)$ is the image of a suitable
embedding of the corresponding model $(\Sigma_{g,s},\partial \Sigma_{g,s})$ in
$S^3$.  As a consequence, $L$ admits a {\it special diagram} $\Dd(T)$
as described in Figure~\ref{surface4:fig}: on the top there is a
suitable decorated tangle $T$ with $2g+s-1$ strings (see
Section~\ref{ribbon}), where we understand that all the strings have
positive sign; on the bottom we see a standard {\it closing tangle}
$C$ which closes the ribbon boundary tangle $D(T)$ associated to $T$.
The strings of $T$ correspond to a generic projection of the image
(via the embedding $\Sigma_{g,s}\hookrightarrow S^3$) of some oriented
edges of the spine $P_{g,s}$ of $\Sigma_{g,s}$. We say that $T$ is the
\emph{primary tangle} of the special diagram $\Dd(T)$.  If $g+s=1$,
then $L$ is a trivial knot and $\Sigma$ is a spanning disk of $L$; in
this case we understand that the only special diagram of $L$ is given
by the trivial diagram $D$ of $L$, and we agree that the closing
tangle $C$ coincides with $D$, while the primary tangle $T$ is empty.

\begin{figure}[htbp]
\begin{center}
\input{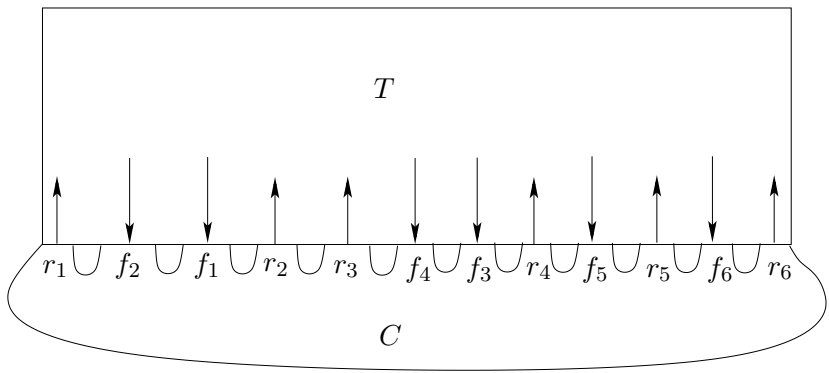} 
\caption{\label{surface4:fig} A special diagram of a $3$--component
  link endowed with a Seifert surface of genus $2$.}
\end{center}
\end{figure}  

\medskip

Let us now consider an oriented link $L$ endowed with a boundary
partition $\Pp=(L_1,\ldots,L_h)$, and let $\Sigma_1,\ldots,\Sigma_h$ be a system
of disjoint Seifert surfaces such that $\partial \Sigma_i=L_i$ 
(as \emph{oriented} $1$--manifolds). If $g_i$ and $s_i$ are the
genus and the number of boundary components of $\Sigma_i$, then the pair
$(\Sigma_1\cup\ldots\cup \Sigma_h,L)$ is the image of a suitable embedding in
$S^3$ of the disjoint union $\bigsqcup_{i=1}^h (\Sigma_{g_i,s_i},\partial
\Sigma_{g_i,s_i})$.  It readily follows that $L$ admits a special diagram
as described in Figure~\ref{surface5:fig}, where the closing tangle
$C$ decomposes into the union of $h$ closing tangles
$C_1,\ldots,C_h$. Of course, strings of $T$ corresponding to distinct
$\Sigma_i$'s may be linked to each other.

Such a special diagram is \emph{adapted} to $\Pp$, in the 
sense that every arc of the primary tangle $T$ gives rise to a pair
of arcs of $D(T)$ that belong to the same link of the partition $\Pp$.
Therefore, every $\Pp$--cycle on $L$ descends to a well--defined cycle
on $T$. 

\begin{figure}[htbp]
\begin{center}
\input{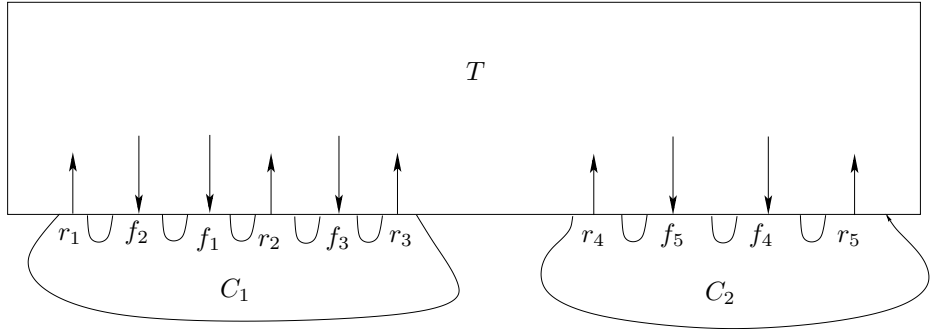}
\caption{\label{surface5:fig} A special diagram for the link
  $L=L_1\cup L_2$, where $L_1$ is a $2$--component link bounding a
  Seifert surface of genus 1, and $L_2$ is a knot bounding a Seifert
  surface of genus 1.}
\end{center}
\end{figure}  

\begin{remark}\label{diagram:rem}
  {\rm Suppose that $\Pp$ is a boundary partition of a $k$--component
    link $L$. The procedure described in this Section provides a
    special diagram of $L$ adapted to $\Pp$ whose primary tangle has
    exactly $2g(\Pp)+k-|\Pp|$ strings.  }
\end{remark}

\section{Lower bounds for link genera}\label{main-dim}
We are now ready to give a
\smallskip

{\it Proof of Theorem~\ref{generalbound}.}  Let $(L,\Pp)$ be a
$k$--component partitioned link, and let us set
$$
\alpha= 2g(\Pp)+k-|\Pp|\ .
$$
As pointed out in Remark~\ref{diagram:rem}, $L$ admits a special diagram 
$\Dd(T)$ adapted to $\Pp$ whose primary tangle $T$
has exactly $\alpha$ strings.

Let us take a quandle $\X \in \Qq_\Ff$ and 
a $\Pp$--cycle $\overline{z}\colon \Pp\to \Z_{t_\X}$.
Such a cycle descends to the diagram $\Dd(T)$, whence to the
boundary ribbon tangle $D(T)\subset \Dd(T)$. What is more, 
since $\Dd(T)$ is adapted to $\Pp$, 
the cycle $\overline{z}$ 
induces a cycle on $T$, which will also be denoted by $\overline{z}$. 
The $\X$-colorings of $(\Dd(T),\overline{z})$ are the
$\X$-colorings of $(D(T),\hat{z})$ that extend to the whole
$(\Dd(T),\overline{z})$. 

In order to study 
the space of $\X$-colorings of $(D(T),\overline{z})$ we exploit
the results obtained in Section~\ref{ribbon}. The space of $\X$-colorings
of $(\Dd(T),\overline{z})$ is then obtained by imposing 
the conditions corresponding to the fact that colors have to match along
the closing tangle $C$ of $\Dd(T)$.

Let us associate to every string of $T$ four variables
$(a_i,d_i,b_i,f_i)$, $i=1,\ldots,\alpha$. As usual, the pair
$(a_i,a_i+d_i)$ refers to the values of a $\X$-coloring
on the arcs of $D(T)$ originating at the input point of the 
$i$-th string of $T$, while $(a_i+f_i,a_i+d_i+f_i)$ refers to the values
of such a coloring on the arcs of $D(T)$ ending at the output point.
Finally, the auxiliary variable $b_i$ encodes
the change that
an arc of $D(T)$ undergoes whenever it undercrosses the band corresponding to the $i$--th string.
Therefore, the value of $b_i$ depends both on $d_i$ and 
on the value assigned by $\overline{z}$ to the $i$-th string of $T$.
Henceforth, we denote such a value by $z_i$ 
(so $z_i=\overline{z}(j(i))$
when the $i$-th string of $T$ corresponds to a band of $D(T)$ 
whose boundary lies on $L_{j(i)}$). 

\smallskip

Let us write down the system that computes the space of colorings we
are interested in.  Proposition~\ref{diff-out} implies that the space of
$\X$-colorings of $(D(T),\overline{z})$ is identified with the space
of the solutions of the linear system
\begin{equation}\label{system1}
b_i =  - \overline{t}^{-z_i}H_{z_i}(\overline{t})d_i,\qquad  i=1,\ldots, \alpha
\end{equation}
\begin{equation}\label{system2}
f_i = \sum_{j=1}^{\alpha} M_{i,j}b_j, \qquad   i=1,\ldots,\alpha\ .
\end{equation}
In order to obtain the space of $\X$-colorings of
$(\Dd(L),\overline{z})$, we have to add to these equations also the
conditions arising from the fact that colors must match along the
strings of the closing tangle $C$. These conditions can be translated
into a linear system
\begin{equation}\label{system3}
S(\{a_i\},\{d_i\},\{f_i\})=0 \ ,
\end{equation}
and we stress  that such a system does not involve the $b_i$'s
(this system 
is written down in Subsection~\ref{diag:det}, 
but this is not relevant to our purposes here).

Let now $\overline{z}'$ be another $\Pp$--cycle, and let us 
concentrate on the difference
$$ |a_\X(L,\Pp,\overline{z}) - a_\X(L,\Pp,\overline{z}')|\ .$$
We have just seen that the linear system that computes the space
$c_\X(L,\Pp,\overline{z})$ is given by the union of
equations~\eqref{system1}, \eqref{system2}, \eqref{system3}.  Now, the
argument above shows that the system computing the space
$c_\X(L,\Pp,\overline{z}')$ is given by the union of the
systems~\eqref{system2} and \eqref{system3} with the following
equations:
\begin{equation}\label{system4}
b_i =  
- \overline{t}^{-z'_i}H_{z'_i}(\overline{t})d_i,\qquad  i=1,\ldots, \alpha\ ,
\end{equation}
where $z'_i$ is the value assigned by $\overline{z}$ to the $i$-th
string of $T$.  Therefore, the system computing
$c_\X(L,\Pp,\overline{z}')$ is obtained from the system computing
$c_\X(L,\Pp,\overline{z}')$ just by replacing
equations~\eqref{system1} with equations~\eqref{system4}.  Since such
equations are in number of $\alpha=2g(\Pp)+k-|\Pp|$ we finally obtain
$$
|a_\X(L,\Pp,\overline{z}) - a_\X(L,\Pp,\overline{z}')|\leq 2g(\Pp)+k-|\Pp|\ .
$$ 
This concludes the proof of Theorem~\ref{generalbound}.

\cvd

Finally we note that the very same argument of the above proof gives
the following improvement of Theorem~\ref{generalbound} .

\begin{teo}\label{generalbound2} 
  Let $\Pp=(L_1,\ldots,L_h)$ be a boundary partition of $L$, where
  $L_i$ is a $k_i$--component link, let $\overline{z}$ and
  $\overline{z}'$ be two $\Pp$--cycles on $L$, and let
  $I=\{i\in\{1,\ldots,h\}\, |\, \overline{z}(i)\neq
  \overline{z}'(i)\}$.  Let also $(\Sigma_1,\ldots,\Sigma_h)$ be a system of
  disjoint Seifert surfaces for the $L_i$'s. Then the following
  inequality holds:
$$
|a_\X (L,\Pp,\overline{z})-a_\X (L,\Pp,\overline{z}')| \leq 
2\sum_{i\in I} g(\Sigma_i) +\sum_{i \in I} k_i - |I|\ .
$$
\end{teo}

\subsection{An example}
The following example shows that Theorem~\ref{generalbound2} could prove more effective
than Theorem~\ref{generalbound} in providing bounds on the genus of links.

Let us consider the tangle $B$ showed in Figure~\ref{bande2:fig}.
Recall that $\Dd_p$ is the dihedral quandle of order $p$,
let $\overline{1}$ be the cycle that assigns the value 
$1\in\Z_{t_{\D_p}}=\Z_2$ to every arc of $B$, 
and let us denote by $C_p(a,b,c,d)$
the number of $\Dd_p$--colorings of $B$ (relative
to the cycle $\overline{1}$) which extend
the colors $a,b,c,d$ assigned on the ``corners'' of the diagram.

\begin{figure}[h]
\begin{center}
\input{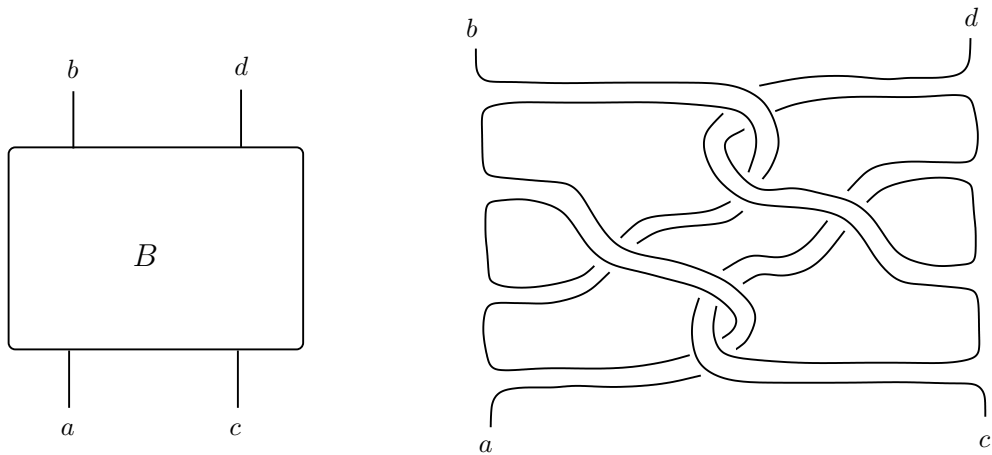}
\caption{\label{bande2:fig} 
The tangle $B$.}
\end{center}
\end{figure}

The following Lemma is proved in~\cite{RR}:

\begin{lem}\label{Blemma}
 We have 
$$
\left\{
\begin{array}{ll}
C_p(a,b,c,d)=p^2\quad &{\rm if}\ a=b,\ c=d\ {\rm and}\ p=3\\
C_p(a,b,c,d)=1\quad &{\rm if}\ a=b,\ c=d\ {\rm and}\ p\neq 3\\
C_p(a,b,c,d)=0\quad & {\rm otherwise}\, .
\end{array}\right.
$$
\end{lem}

For every $q\geq 1$, let $L_q$ be the link described in Figure~\ref{bande4:fig}.

\begin{figure}[h]
\begin{center}
\input{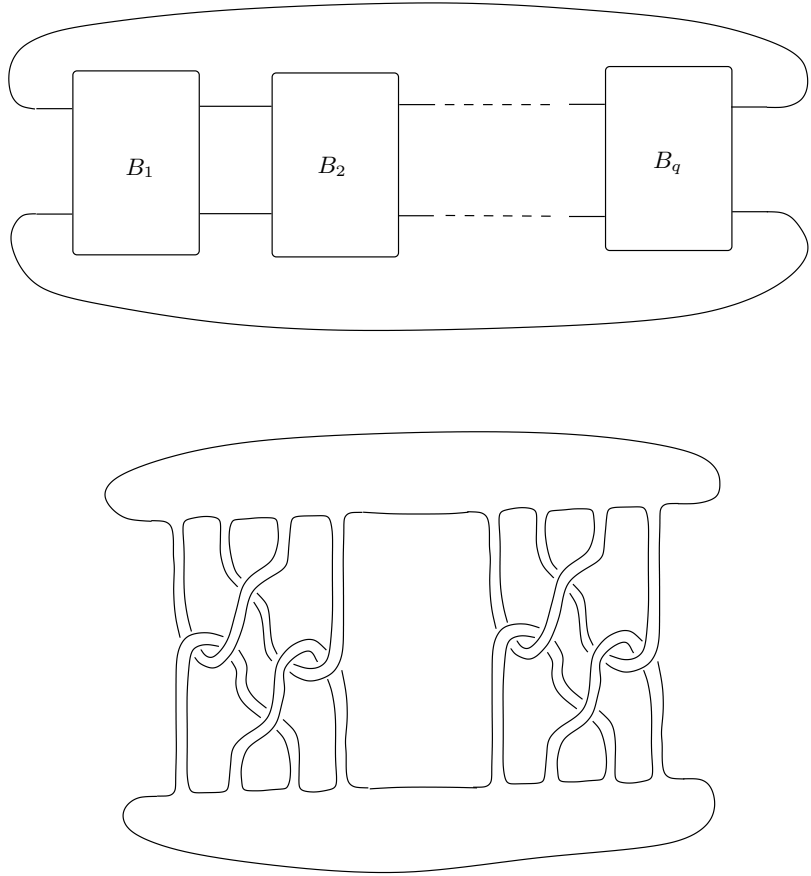}
\caption{\label{bande4:fig} 
On the top: the link $L_q$; every $B_i$ is a copy of the tangle $B$.
On the bottom: the case $q=2$.}
\end{center}
\end{figure}

It is obvious from the picture that $L_q$ is a boundary link such that
$g(L,\Pp_M)\leq 2q$. Let $K_q$ (resp.~$K'_q$) be the component
of $L_q$ on the top half (resp.~the bottom half) of the diagram
shown on the top of Figure~\ref{bande4:fig}. We denote every
$\Pp_M$--cycle $\overline{z}\colon\Pp_M\to\Z_2$ simply by the pair
$(\overline{z}(\{K_q\}),\overline{z}(\{K'_q\}))$, and the integers
$a_{\D_3}(L_q,\Pp_M,(z_1,z_2))$ simply by
$a_{\D_3}(L_q,(z_1,z_2))$.

\begin{prop}\label{H4}
For every $q\geq 1$ we have
$$
a_{\Dd_3}(L_q,(1,1))=2q+1,\quad a_{\Dd_3} (L_q,(1,0))=a_{\Dd_3}(L_q,(0,1))=
a_{\Dd_3}(L_q,(0,0))=1\ .
$$
\end{prop}
\begin{proof}
As usual, the only $(0,0)$-colorings of $L_q$ are those which are constant
on every component of $L_q$, so $a_{\D_3}(L_q,(0,0))=1$.

Let us now concentrate on $(1,0)$-colorings of $L_q$.
It is immediate to observe that $K_q$ and $K'_q$ are both trivial.
Since the cycle $(1,0)$ vanishes
on $K'_q$,
it is immediate to realize that  any such coloring
restricts to a coloring of $K_q(1)$. Since $K_q$ is trivial, this implies
that every $(1,0)$--coloring of $L_q$ is constant on $K_q$. 
The discussion in Section~\ref{ribbon} now implies that  
that the colorings of $K'_q$ are not affected by the crossings between the bands of $K'_q$
and the bands of $K_q$. Then, every $(1,0)$-coloring
of $L_q$ restricts to a $0$-coloring (\emph{i.e.}~to a constant coloring)
of $K'_q$.
We have proved 
that
the only $(1,0)$-colorings are the ones which are constant on every component
of $L_q$, so $a_{\D_3}(L,(1,0))=1$.
The same is true (by the very same argument) also for $(0,1)$-colorings,
so $a_{\D_3}(L,(0,1))=1$.

Let us now fix two colors $a,b\in\D_3$. An easy application of
Lemma~\ref{Blemma} shows that the number of the colorings 
of $L_q$ which take the value $a$ (resp.~$b$) on the arc of $K_q$
(resp.~of $K'_q$) joining the tangles $B_1$ and $B_q$ is equal
to $3^{2q}$. Therefore, the number of $(1,1)$--colorings of $L_q$ is equal
to $3^{2q+2}$, whence 
the conclusion.
\end{proof}

\begin{cor}\label{esempioboundary}
For every $q\geq 1$ we have
$$
g(L_q,\Pp_M)=2q\ .
$$
\end{cor}
\begin{proof}
Let $(\Sigma_q,\Sigma'_q)$ be a system of disjoint
 Seifert surfaces for $K_q, K'_q$. We have to show that
$g(\Sigma_q)+g(\Sigma'_q)\geq 2q$. By Theorem~\ref{generalbound2} we have
$$
\begin{array}{lllll}
2q&=&|a_{\D_3}(L_q,(1,1))-a_{\D_3}(L_q,(0,1))|&\leq &2g(\Sigma_q)\ ,\\
2q&=&|a_{\D_3}(L_q,(1,1))-a_{\D_3}(L_q,(1,0))|&\leq &2g(\Sigma'_q)\ ,
\end{array}
$$
so $g(\Sigma_q)\geq q$ and $g(\Sigma'_q)\geq q$, whence the conclusion.
\end{proof}

It is maybe worth mentioning that the bound provided by Corollary~\ref{lowerbound}
is less effective in order to compute $g(L,\Pp_M)$. In fact, Proposition~\ref{H4}
implies that $\Aa_\Qq(L,\Pp_M)=2q+1$, so the inequality
$\Aa_\Qq(L,\Pp_M)\leq 2g(L,\Pp_M)+1$ only implies $g(L,\Pp_M)\geq q$.

\section{A proof of Theorem~\ref{inoue:teo}}
With notations as in the preceding Section, 
let us describe more explicitly
the system computing 
the $\X$--colorings
of $(L,\Pp_m,\overline{z})$, where $\X=\F(p,h(t))$ is a quandle in $\Qq_\Ff$.

\subsection{More details on the system associated to a special diagram}\label{diag:det}
Let us now concentrate on the case $\Pp=\Pp_m$, so that there exists
$z\in\N$ such that $z=\overline{z}(i)$ for every $i=1,\ldots,k$
(recall that $k$ is the number of components of $L$).
We also set $g=g(L)=g(L,\Pp_m)$.

Then, the linear system described by equations~\eqref{system1},
\eqref{system2} and \eqref{system3} reduces to the system
\begin{equation}\label{sistemainoue}
f_i = (1-\overline{t}^{-z})\sum_{j=1}^{\alpha} M_{i,j}d_j, \qquad
S(\{a_i\},\{d_i\},\{f_i\})=0 \ ,
\end{equation}
where $\alpha=2g+k-1$ and 
$a_i,d_i,f_i$
have to be considered as variables in $\F(p,h(t))$. Moreover, the system
$S(\{a_i\},\{d_i\},\{f_i\})=0$ has integer coefficients.

Let us look more closely to the closing conditions $S\{a_i\},\{d_i\},\{f_i\})=0$.
By looking at the definition of special diagram for $L$, one can easily show that such closing conditions reduce, after easy simplifications, to the system
$$
\left\{
\begin{array}{ll}
a_{2i-1}=a_{2i}-d_{2i-1},\ 
d_{2i}=f_{2i-1},\
f_{2i}=-d_{2i-1},\quad & i=1,\ldots,g,\\
a_{2i}=a_{2i+1}+d_{2i+1},\quad & i=1,\ldots,g-1,\\
a_i=a_{2g},\ f_i=0,\quad & i=2g+1,\ldots,\alpha\\
a_{2g}=a_1+d_1\quad &   
\end{array}\right. \quad . 
$$
An easy inductive argument shows that the condition $a_{2g}=a_1+d_1$ is a consequence
of equations $a_{2i-1}=a_{2i}-d_{2i-1}$, $i=1,\ldots,g$, and 
$a_{2i}=a_{2i+1}+d_{2i+1}$, $i=1,\ldots,g-1$. 
Therefore, the system~\eqref{sistemainoue} 
is equivalent to the system
\begin{equation}\label{ino2}
\left\{
\begin{array}{ll}
a_{2i-1}=a_{2i}-d_{2i-1},\ & i=1,\ldots,g,\\
a_{2i}=a_{2i+1}+d_{2i+1},\ & i=1,\ldots,g-1,\\
a_i=a_{2g},\ & i=2g+1,\ldots,\alpha \\
-d_{2i}+(1-\overline{t}^{-z})\sum_{j=1}^\alpha M_{2i-1,j}d_j=0,\ & i=1,\ldots,g,\\
d_{2i-1}+(1-\overline{t}^{-z})\sum_{j=1}^\alpha M_{2i,j}d_j=0,\ & 
i=1,\ldots,g,\\
(1-\overline{t}^{-z})\sum_{j=1}^\alpha M_{i,j}d_j=0,\ & i=2g+1,\ldots,\alpha\ ,
\end{array}
\right. 
\end{equation}
where we have eliminated the $f_i$'s from the variables.

Let us now define two square matrices $N(z)$ and $J$
of order $\alpha$ with coefficients in $\Lambda$
as follows:
$$
J_{i,j}=\left\{\begin{array}{ll}
 -1 & {\rm if}\ i=2h-1,\ j=2h,\ h\leq 2g\\
 1 & {\rm if}\ i=2h,\ j=2h-1,\ h\leq 2g\\
 0 & {\rm otherwise}\end{array}
 \right.\qquad
,
$$
(so $J$ has in fact integer coefficients), and
$$
N(z)=(t^z-1)M + t^z J\ .
$$

We also denote by $N(z,p)$ the matrix obtained by replacing each coefficient
of $N(z)$ by its image via $\pi_p\colon \Lambda\to\Lambda_p$, and by
$N(z,p,h(t))$ the matrix obtained by further projecting each
coefficient of $N(z,p)$ onto $\F(p,h(t))$. 

We are now ready to prove the following:

\begin{lem}\label{Qlemma}
The space of 
$\X$--colorings of $(L,\Pp_m,\overline{z})$ is in natural bijection
with 
the direct sum
$$
\F(p,h(t))\oplus \ker N(z,p,h(t))\ ,
$$ 
so
$$
a_\X(L,\Pp_m,\overline{z})=\dim \ker N(z,p,h(t))\ .
$$
\end{lem}
\begin{proof}
The previous discussion shows that the space of colorings we are considering
is in natural bijection with the solutions of the system~\eqref{ino2}.
It is immediate to realize that, for every such solution,
each
$a_i$, $i\geq 2$, is uniquely determined by $a_1$ and the $d_j$'s. 
Moreover, once a solution of the system~\eqref{ino2} is fixed, 
we can obtain another solution just by adding 
a constant term to every $a_i$. 

Therefore, the space of the solutions of~\eqref{ino2}
is isomorphic to the direct sum
of $\F(p,h(t))$ with the space of the solutions of the system
$$
\begin{array}{ll}
-d_{2i}+(1-\overline{t}^{-z})\sum_{j=1}^\alpha M_{2i-1,j}d_j=0,\quad & i=1,\ldots,g,\\
d_{2i-1}+(1-\overline{t}^{-z})\sum_{j=1}^\alpha M_{2i,j}d_j=0,\quad &
i=1,\ldots,g,\\
(1-\overline{t}^{-z})\sum_{j=1}^\alpha M_{i,j}d_j=0,\quad & i=2g+1,\ldots,\alpha\ .
\end{array}
$$
The matrix encoding this system is equal to $\overline{t}^{-z}N(z,p,h(t))$, and $\overline{t}^{-z}$
is invertible in $\F(p,h(t))$,
whence the conclusion.
\end{proof}

\subsection{Some relations between $M$ and the Seifert matrix of 
$L$}
Let $\Sigma$ be the Seifert surface of $L$ encoded by the fixed special diagram $\Dd (T)$ we are considering,
and observe that each (oriented) string of $T$ canonically defines an (oriented) arc lying 
on $\Sigma$. 
The module $H_1(\Sigma;\Z)$ admits a special geometric basis 
$\{\beta_1,\dots,\beta_{\alpha}\}$, where $\beta_j$ is obtained by closing the
$j$--th string of $T$ in the portion of $\Sigma$ carried by the closing
tangle $C$, in such a way that we introduce just one intersection
point between $\beta_{2i-1}$ and $\beta_{2i}$, $i=1,\dots,\
g$, while $\beta_i$ is disjoint from $\beta_j$ for every $i=2g+1,\ldots,\alpha$, $j=1,\ldots,\alpha$. 
Recall that the Seifert matrix $S(L)$ of $L$ is the square matrix with integer coefficients
defined by $S(L)_{i,j}={\rm lk} (\beta_i,\beta_j^+)$, $i,j=1,\ldots,\alpha$, where
${\rm lk} (\beta_i,\beta_j^+)$ is 
the linking number (in $S^3$) between 
$\beta_i$ and the knot $\beta_j^+$ obtained by slightly pushing $\beta_j$ to the positive side of $\Sigma$.
>From the very definition of linking number we readily obtain the following:

\begin{lem}\label{seifert}
We have
$$
S(L)=\frac{M+M^T+J}{2}\ .
$$
\end{lem}

Let us point out another interesting property of $M$ that will prove useful later:

\begin{lem}\label{B==1}
We have 
$$
M-M^T=-J\ .
$$
\end{lem}
\begin{proof}
Since both $M-M^T$ and $J$ are antisymmetric, it is sufficient to show that
for every $i<j$ we have
$M_{i,j}-M_{j,i}=1$ if $j=i+1$, and $M_{i,j}-M_{j,i}=0$ otherwise.
However, it follows from the definition of $M$ that the number $M_{i,j}-M_{j,i}$
is equal to the algebraic intersection number between the projections of
the $j$--th and the $i$--th string of $T$ 
(taken in this order) onto the plane containing the special diagram.
If $j>i+1$ (resp.~$j=i+1$), such number is equal to the algebraic intersection number between the projections of
$\beta_j$ and of $\beta_i$ (resp.~is equal to $1$ plus the algebraic intersection number between the projections of
$\beta_j$ and of $\beta_i$). But the  algebraic intersection number between the projections of
$\beta_j$ and of $\beta_i$ is obviously null, whence the conclusion.
\end{proof}

Putting together Lemmas~\ref{seifert} and~\ref{B==1} we get the following:

\begin{cor}\label{seifertcor}
 We have
$$
S(L)=M+J,\qquad t^zS(L)-S(L)^T=(t^z-1)M+t^zJ=N(z)\ . 
$$
\end{cor}
 
\subsection{Proof of Theorem~\ref{inoue:teo}}\label{proofinoue}


Let $\widetilde{X}(L)$ and $A^{(p)} (L)$ be the cyclic covering and
the $\Lambda_p$--module defined in the Introduction.
We have the following: 

\begin{lem}\label{reduce:lemma}
The module $A^{(p)} (L)$ admits the square presentation matrix
$$
tS(L)^{(p)}-(S(L)^{(p)})^T= N(1,p)\ .
$$ 
In particular, $\Delta^{(p)}_i (L)(t)=e^{(p)}_i (L)(t)=1$ for every
$i>\alpha$.
\end{lem}
\begin{proof}
The usual proof that $tS(L)-S(L)^T$ is a presentation
of $H_1 (\widetilde{X}(L);\Z)$ over $\Lambda$ 
relies on some standard Mayer--Vietoris argument and on 
Alexander--Lefschetz duality, which ensures 
that, if $\{\beta_1,\ldots,\beta_{2g+k-1}\}$ is any base of the first homology group of
a Seifert surface $\Sigma$ for $L$, then the first homology group of $S^3\setminus \Sigma$ admits
a dual base $\{\gamma_1,\ldots,\gamma_{\alpha}\}$ such that ${\rm lk} (\beta_i,\gamma_j)=\delta_{ij}$ (see \emph{e.g.}~\cite[Chapter 8]{burde}).
Both these tools may still be exploited  
when $\Z$ is replaced by $\Z_p$,
and this readily implies the conclusion.

An alternative proof can be obtained as follows. An easy application of the Universal
Coefficient Theorem for homology shows that $A_p (L)\cong A(L)\otimes_\Z \Z_p$,
and this easily implies that any presentation
$$
\xymatrix{
0\ar[r] & \Lambda^r\ar[r] & \Lambda^s \ar[r] & A(L) \ar[r] & 0}
$$
induces a presentation
$$
\xymatrix{
\Lambda_p^r \ar[r] \ar@{<->}[d]
& \Lambda_p^s \ar[r] \ar@{<->}[d] & A_p(L)\ar@{<->}[d]\ar[r] & 0\\
\Lambda^r\otimes_\Z \Z_p 
& \Lambda^s\otimes_\Z \Z_p & A(L)\otimes_\Z \Z_p\ ,}
$$
whence the conclusion.
\end{proof}

The following result describes some relations between 
$\Delta_i(L)(t)$ and $\Delta^{(p)}_i(L)(t)$, where $i\in\N$.

\begin{cor}\label{princ:cor}
\begin{enumerate}
 \item 
For every $i\in\N$ we have
$E_i^{(p)} (L)=\pi_p (E_i (L))$. 
\item
For every $i\in\N$ the polyonomial
$\pi_p(\Delta_i(L)(t))$ divides $\Delta^{(p)}_i (L)(t)$
in $\Lambda_p$.
\item
 We have
$
\Delta^{(p)}_1(L)(t)=\pi_p (\Delta (L)(t))
$.
\item
If $f(t)\in E_i(L)$, then  
$\Delta_i^{(p)}(L)(t)$ divides $\pi_p(f(t))$ in $\Lambda_p$.
\end{enumerate}
\end{cor}
\begin{proof}
By Lemma~\ref{reduce:lemma}, $\pi_p$ maps a set of generators
(over $\Lambda$) of the ideal $E_i(L)$ 
onto a set of generators (over $\Lambda_p$) of the ideal
$E^{(p)}_i (L)$. Since $\pi_p$ is surjective, this 
readily implies point~(1).

By point~(1), the polynomial $\pi_p(\Delta_i(L)(t))$
divides every element of $E_{i-1}^{(p)}(L)$, whence point~(2).

Since $A(L)$ admits the square presentation matrix
$S(L)-tS(L)^T$, the ideal
$E_0 (L)$ is principal.
Together with~(1), this immediately gives~(3).

Point~(4) is an easy consequence of point~(1).
\end{proof}

Let us now consider the $\Lambda_p$--linear map
$\psi_z\colon \Lambda_p^\alpha\to 
\Lambda_p^\alpha$ such that $\psi_z(x)= N(z,p)\cdot x$
for every $x\in \Lambda_p^\alpha$. Of course,
$\psi_z$ induces a quotient map $\overline{\psi}_z\colon \F(p,h(t))^\alpha\to
\F(p,h(t))^\alpha$ such that $\overline{\psi}_z(\overline{x})=N(z,p,h(t))\cdot \overline{x}$
for every $\overline{x}\in \F(p,h(t))^\alpha$.
Let now $\overline{z}$ be a $\Pp$--cycle for $L$, and denote by
$z=\overline{z}(\Pp_m)$ the value assigned by $\overline{z}$
to every component of $L$. 
By Lemma~\ref{Qlemma}, the space of $\X$--colorings of $(L,\Pp_m,\overline{z})$ is in bijection
with $\F(p,h(t))\oplus \ker \overline{\psi}_z$, whence to $\F(p,h(t))\oplus
{\rm coker}\, \overline{\psi}_z$ (here we use that $\F(p,h(t))^\alpha$ is finite).

Since $\Lambda_p$ is a P.I.D.,
there exist square univalent matrices $U(1),V(1)$ with coefficients in $\Lambda_p$
such that 
$$
U(1)\cdot N(1,p)\cdot V(1)={\rm diag} \left(e^{(p)}_1(L) (t),\ldots, e^{(p)}_\alpha (L)(t)\right)\ ,
$$
where
${\rm diag} (\gamma_1,\ldots,\gamma_\alpha)$ denotes
the diagonal matrix with the $\gamma_i$'s on the diagonal.
Let $U(z)$ (resp.~$V(z)$) be the matrix obtained by applying to every coefficient
of $U(1)$ (resp.~$V(1)$) the ring endomorphism of $\Lambda_p$ 
that maps $t$ to $t^z$. 
Then we obviously have
$$
U(z)\cdot N(z,p)\cdot V(z)={\rm diag} \left(e^{(p)}_1(L) (t^z),\ldots, e^{(p)}_\alpha (L)(t^z)\right)\ .
$$ 
After reducing the coefficients modulo $h(t)$, this equality translates
into the equality
$$
\overline{U}(z)\cdot N(z,p,h(t))\cdot \overline{V}(z)={\rm diag} \left({e^{(p)}_1}(L) (t^z),\ldots, {e^{(p)}_\alpha}(L) (t^z)\right)\ ,
$$
where $\overline{U}(z),\overline{V}(z)$ are invertible over $\F(p,h(t))$,
and we denote the class of $e^{(p)}_i(L)(t)$ in $\F(p,h(t))$
simply by  ${e^{(p)}_i}(L) (t)$.
This implies that
${\rm coker}\, \overline{\psi}_z$ is isomorphic to
$$
\bigoplus_{i=1}^\alpha \F(p,h(t))\Big/\left({e_i^{(p)}}(L)(t^z)\right) \cong
\bigoplus_{i=1}^\alpha \Lambda_p\Big/\left(e_i^{(p)}(L)(t^z),h(t)\right)
\cong 
\bigoplus_{i=1}^\infty \Lambda_p\Big/\left(e_i^{(p)}(L)(t^z),h(t)\right)\ ,
$$
where the last equality is due to Lemma~\ref{reduce:lemma}.
This concludes the proof of Theorem~\ref{inoue:teo}. For later purposes we point out the
following easy:

\begin{cor}\label{a>1}
 We have
$$
a_\X(L,\Pp_m,\overline{z})>0
$$
if and only if $h(t)$ divides $\pi_p (\Delta(L)(t^z))$ in $\Lambda_p$.
\end{cor}
\begin{proof}
 Theorem~\ref{inoue:teo} implies that $a_\X (L,\Pp_m,\overline{z})>0$ if and only if
$h(t)$ divides $e_i^{(p)}(L)(t^z)$ for some $i\geq 1$. The conlcusion follows from
the fact that 
$$
\pi_p (\Delta(L)(t^z))=\Delta_1^{(p)}(L)(t^z)=\prod_{i=1}^{\infty} e_i^{(p)}(L)(t^z)\ .
$$
\end{proof}

\begin{remark}\label{inoue:rem}
  {\rm On may wonder if the equality $\pi_p (e_i(L)(t))\doteq
    e_i^{(p)}(L)(t)$ holds for every $i\geq 1$, so that in the statement of
    Theorem~\ref{inoue:teo} we could replace the
    summand $$\Lambda_p/(e_i^{(p)}(L)(t^z),h(t))$$ with the module
$$\Lambda_p/(\pi_p(e_i(L)(t^z)),h(t))\ .$$ 
Such a claim seems also suggested, at least when $L$ is a knot and
$\overline{z}=1$, by the original statement of~\cite[Theorem
1]{Inoue}.  However, this is not the case, as the following
construction shows.

In fact,
let $k_1(t)=t-1+t^{-1}$ and $k_2(t)=-2t+5-2t^{-1}$, and 
observe that $k_i(t^{-1})=k_i(t)$, $i=1,2$. It is proved 
in~\cite[Theorem 2.5]{kearton} that a knot $K$ exists such that
$A(K)$ is presented by the matrix
$$
{\rm diag}\, \left(k_1(t),k_1(t),{k}_2(t),{k}_2 (t)\right)\ .
$$
This readily implies that
$$
\begin{array}{ll}
E_0 (K)=(k_1(t)^2k_2(t)^2),\qquad & E_1(K)=(k_1(t)^2k_2(t),k_1(t)k_2(t)^2),\\
E_2(K)=(k_1(t)^2,k_2(t)^2,k_1(t)k_2(t)),\qquad & E_3(K)=(k_1(t),k_2(t))
\end{array}
$$
and $E_i(K)=\Lambda$ for every $i\geq 4$, 
whence
$$
\Delta_1 (K)(t)=k_1(t)^2k_2(t)^2,\qquad  \Delta_2(K)(t)=k_1(t)k_2(t),
$$
and $\Delta_i(K)(t)=1$ for every $i\geq 3$.
As a consequence we get 
$$
e_1(K)(t)=e_2(K)(t)=k_1(t)k_2(t),\qquad e_i(K)(t)=1\ {\rm for\ every}\ i\geq 2\ .
$$
On the other hand, let us fix $p=3$,
and observe that in this case $\pi_3 (k_1(t))=\pi_3(k_2(t))=k(t)\in\Lambda_3$,
where $k(t)=(t+1)^2$.
Therefore, from the equality 
$E^{(3)}_{i}(K)=\pi_p (E_{i}(K)(t))$ (see Corollary~\ref{princ:cor}) we easily deduce that
$$
\Delta^{(3)}_1 (K)=k(t)^4,\quad  \Delta^{(3)}_2(K)=k(t)^3,\quad
\Delta^{(3)}_3(K)=k(t)^2,\quad  \Delta_4^{(3)}(K)=k(t),
$$
and $\Delta^{(3)}_i(K)=1$ for every $i\geq 5$, 
so
$$
e^{(3)}_1(K)(t)=e^{(3)}_2(K)(t)=e^{(3)}_3(K)(t)=e^{(4)}(K)(t)=k(t), 
\qquad e^{(3)}_i(K)(t)=1\ {\rm for\ every}\ i\geq 5\ .
$$
Therefore, if $h(t)=t+1\in\Lambda_3$, then we have
$$
\bigoplus_{i=1}^\infty \Lambda_3/(e_i^{(3)}(K)(t),h(t))\cong \F(3,h(t))^4\ ,
$$
while
$$
\bigoplus_{i=1}^\infty \Lambda_3/(\pi_3(e_i(K)(t)),h(t))\cong \F(3,h(t))^2\ .
$$
}
\end{remark}

\begin{remark}\label{controes}
{\rm
Let $k_1(t),k_2(t)\in\Lambda$ and $k(t)\in\Lambda_3$ be the polynomials introduced in the previous Remark. It is proved in~\cite{levine}
that a knot $K'$ exists whose module $A(K')$ is isomorphic to $\Lambda/(k_1(t)k_2(t))$
(see also~\cite[Theorem 7.C.5]{rolfsen}). 
Let $K''=K'+K'$. Then we have $A(K'')=A(K')\oplus A(K')$
(see \emph{e.g.}~\cite[Theorem 7.E.1]{rolfsen}), and this readily implies
that 
$$
E_0(K'')=(k_1(t)^2k_2(t)^2),\qquad E_1(K'')=(k_1(t)k_2(t))\ ,
$$
and $E_i (K'')=\Lambda$ for every $i\geq 2$.
Therefore,
$$
\Delta_1(K'')(t)=k_1(t)^2k_2(t)^2,\qquad \Delta_2(K'')(t)=k_1(t)k_2(t),
 $$
 and $\Delta_i(K'')(t)=1$ for every $i\geq 3$. 
 Moreover, since the elementary ideals of $K''$ are principal, we also have
 $$
\Delta^{(3)}_1(K'')(t)=\pi_3(\Delta_1(K)(t))=k(t)^4,\qquad \Delta_2^{(3)}(K'')(t)=\pi_3(\Delta_2(K)(t))=k(t)^2,
 $$
 and $\Delta_i^{(3)}(K'')(t)=\pi_3(\Delta_i(K)(t))=1$ for every $i\geq 3$. 
 
 Therefore, the knot $K''$ and the knot $K$ introduced in the previous Remark satisfy
 the condition $\Delta_i(K)(t)=\Delta_i(K'')(t)$ for every $i\geq 1$, but have a different
 number of $\Dd_3$--colorings with respect to the cycle $\overline{z}=\overline{1}$.
What is more, since for every $p$ the $\Lambda_p$--module
$A^{(p)}(K)$ (resp.~$A^{(p)}(K'')$) admits a square presentation matrix
of order $4$ (resp.~of order $2$), Theorem~\ref{inoue:teo} readily implies
that $\Aa_\Qq(K)\leq 4$ (resp.~$\Aa_\Qq(K'')\leq 2$). Our computations
imply now that $\Aa_\Qq(K)=4$ and $\Aa_\Qq(K'')= 2$. 
Therefore, even if they
share every Alexander polynomial
$\Delta_i(K)(t)=\Delta_i(K'')(t)$, $i\geq 1$, the knots $K$, $K''$
are distinguished from each other
by the invariant $\Aa_\Qq$.
}
\end{remark}


\section
{Comparing $\Aa_\Qq$ with $\Aa$}\label{AvsA}
Let us keep notation from the preceding Section. 
Of course, since $\F(p,h(t))$ is a field, 
the quotient $\Lambda_p/(e^{(p)}_i(L)(t^z),h(t))$ 
of $\F(p,h(t))$ 
is null
(resp.~isomorphic to $\F(p,h(t))$) if and only if
$h(t)$ divides (resp.~does not divide) $e^{(p)}_i(L)(t^z)$
in $\Lambda_p$.
Therefore, if we set
$$
I(z,p,h(t),L)= \{i\in\N^*\, |\ h(t)\ {\rm divides}\ e_i^{(p)}(L)(t^z) \}\ ,
$$
$$|I(z,L)| = \sup_{p,h(t)} |I(z,p,h(t),L)| \ , $$
$$|I(L)|=\sup_z |I(z,L)|\ ,$$
then we easily obtain that
$$
\Aa_\Qq(L)= |I(L)| \ .
$$

Therefore, in order to prove Theorem~\ref{K=K} it is sufficient
to show that, if $L=K$ is a knot, then:
\begin{itemize}
\item
$I(K)\leq \Aa (K)$,
\item
$I(K)=0$ if and only if $\Aa (K)=0$.
\end{itemize}


\subsection{Reduction to the cycle $\overline{z}=1$}
We first prove that,
in order to compute $\Aa_\Qq(K)$, it is sufficient to restrict
our attention to colorings relative to the cycle $\overline{z}=\overline{1}$.

\begin{lem}\label{z=1}
We have
$$
I(L)=I(1,L)\ .
$$
\end{lem}

In the proof we use the following elementary

\begin{lem}\label{elementary}
  Let $p_1(t),\ldots,p_n(t)$ be polynomials in $\Lambda_p$, and
  let $d(t)$ be their G.C.D.~in $\Lambda_p$. For every integer
  $z\geq 1$, the polynomial $d(t^z)\in\Lambda_p$ is the G.C.D.~of
  $p_1(t^z),\ldots,p_n(t^z)$ in $\Lambda_p$.
\end{lem}
\begin{proof}
  For every $i=1,\ldots,n$, the fact that $d(t)$ divides $p_i (t)$
  readily implies that $d(t^z)$ divides $p_i(t^z)$. On the other hand,
  $\Lambda_p$ is P.I.D., so Bezout's Identity implies
  that there exist $\lambda_1(t),\ldots,\lambda_n (t)\in\Lambda_p$ such
  that
$$
d(t)=\lambda_1 (t)p_1(t)+\ldots+\lambda_n(t) p_n(t)\ ,
$$
whence
$$
d(t^z)=\lambda_1 (t^z)p_1(t^z)+\ldots+\lambda_n(t^z) p_n(t^z)\ .
$$
Therefore, if $d'(t)$ divides every $p_i(t^z)$, then $d'(t)$
also divides $d(t^z)$, whence the conclusion.
\end{proof}

{\it Proof of Lemma \ref{z=1}.}
It is sufficient to show that, for every odd prime $p$, every positive
integer $z$ and every  irreducible
polynomial $h(t)\in\Lambda_p$, there exists an irreducible polynomial
$h'(t)\in\Lambda_p$ such that 
$$
| I(z,p,h(t),L)|\leq |I(1,p,h'(t),L)|\ .
$$

Let $d(t)\in\Lambda_p$ be the G.C.D.~of the polynomials $\{e_i^{(p)}(t),\ i\in I(z,p,h(t),L)\}$.
By the very definitions, $h(t)$ divides $e_i^{(p)}(t^z)$ for every
$i\in I(z,p,h(t),L)$, so by Lemma~\ref{elementary} we have that
$h(t)$ divides $d(t^z)$. This implies that
the breadth of $d(t)$ is positive, 
so 
$d(t)$ admits an irreducible factor $h'(t)$ of positive breadth.
By construction we have that $h'(t)$ divides $e_i^{(p)}(t)$ for every
$i\in I(z,p,h(t),L)$, so $I(z,p,h(t),L)\subseteq I(1,p,h'(t),L)$, whence the conclusion.
\cvd

\subsection{More details on Alexander ideals of links}
Recall that $\widetilde{X}(L)$ is the total linking number covering of the complement
of $L$, and that $k$ denotes the number of components of $L$.
If $x_0\in \compl (L)$ is any basepoint and $\widetilde{X}_0$ is the preimage of
$x_0$ in $\widetilde{X}(L)$, then the relative homology module
$A' (L)=H_1(\widetilde{X}(L),\widetilde{X}_0;\Z)$ also admits a natural structure of
$\Lambda$--module. Moreover, it is not difficult to show that
$A'(L)\cong A(L)\oplus \Lambda$ (as $\Lambda$--modules), so
$E_i(A(L))=E_{i+1} (A'(L))$ for every $i\in\N$, and $\Delta_i (L)(t)\in\Lambda$
is the generator of the smallest principal ideal containing $E_{i} (A'(L))$. If $L=K$ is a knot, this immediately
implies that $\Delta_i (L)(t)\in\Lambda$ concides with the so called
\emph{$i$--th Alexander polynomial of $K$}.

\begin{lem}\label{k-unitary}
We have
$$
\Delta^{(p)}_{i}(L)(t)\neq 0\qquad {\rm for\ every}\ i\geq k\ .
$$ 
\end{lem}
\begin{proof}
Let $\widehat{X} (L)$ be the maximal abelian covering of 
$\compl (L)$ and let $\widehat{X}^0 $ be the preimage
of $x_0$ in $\widehat{X} (L)$.
Then the homology group $\widehat{A} (L)=H_1 (\widehat{X} (L),
 \widehat{X}^0;\Z)$ admits a natural structure of
$\Z[t_1,t_1^{-1},\ldots,t_k,t_k^{-1}]$--module (see \emph{e.g.}~\cite{hill,hill2}).
Just as in the case of the total linking number covering,
one may define the $i$--th elementary ideal $\widehat{E}_i (L)\subseteq \Z[t_1,t_1^{-1},\ldots,t_k,t_k^{-1}]$ of 
this module. If $\tau\colon \Z[t_1,t_1^{-1},\ldots,t_k,t_k^{-1}]\to \Lambda$
is the ring homomorphism that sends each $t_i^{\pm 1}$ into $t^{\pm 1}$, it is not difficult
to show that 
$$
E_i(L)=E_{i+1}(A'(L))=\tau (E_{i+1} (\widehat{A}(L)))
$$
(see \emph{e.g.}~\cite[page 106]{hill2}). 

Let now $\varepsilon \colon \Z[t_1,t_1^{-1},\ldots,t_k,t_k^{-1}]\to \Z$ be the 
\emph{augmentation homomorphism} defined by $\varepsilon (f(t_1,\ldots,t_k))=f(1,\ldots,1)$.
A classical result about Alexander ideals of links (see \emph{e.g.}~\cite[Lemma 4.1]{hill2})
ensures that
$\varepsilon (E_k(\widehat{A}(L)))=\Z$.
In particular, 
there exists 
$g(t)\in E_k (\widehat{A}(L))$ such that $g(1,\ldots,1)=1$. 
Let us set $f(t)=\tau (g(t))\in 
E_{k-1}(L)$ and $f^{(p)}(t)=\pi_p (f(t))$. Our choices readily imply that $f^{(p)}(1)=1$
in $\Z_p$, so $f^{(p)}(t)\neq 0$ in $\Lambda_p$.
By Corollary~\ref{princ:cor}--(4),
the polynomial $\Delta^{(p)}_{k} (L)(t)$ divides $f^{(p)}(t)$ in $\Lambda_p$,
so $\Delta^{(p)}_i(t)$ is not null for every $i\geq k$.
\end{proof}

\begin{cor}\label{k-unitary:cor}
We have
$$
e^{(p)}_{i}(L)(t)\neq 0\qquad {\rm for\ every}\ i\geq k\ .
$$ 
\end{cor}


\subsection{Proof of Theorem~\ref{K=K}}
The key step for proving Theorem~\ref{K=K} is the following:

\begin{prop}\label{Alink}
 If $L$ is a $k$--component link, then
$$
I(1,p,h(t),L)\leq \frac{\br \Delta_{k}^{(p)}(L)(t)}{\br h(t)} + k -1 \ .
$$
\end{prop}
\begin{proof}
By the very definitions we have
$$
\Delta^{(p)}_{k} (L)(t)=\prod_{i\geq k} e^{(p)}_i (L)(t),\
$$
so (since $\Delta_k^{(p)}(L)(t)\neq 0$ by Lemma~\ref{k-unitary})
$$
\br \Delta^{(p)}_{k} (t)=\sum_{i\geq k} \br e^{(p)}_i (L)(t)\ .
$$
Let us now set $I'(1,p,h(t),L)=I(1,p,h(t),L)\cap \{i\in\N\, |\, i\geq k\}$. 
In order to conclude it is sufficient to show that 
$$
|I' (1,p,h(t),L)|\leq \frac{\br \Delta^{(p)}_{k} (t)}{ \br h(t)}\ .
$$

By Corollary~\ref{k-unitary:cor},
if $i\in I'(1,p,h(t),L)$ then $e^{(p)}_i (L)(t)$ is not null and divisible by $h(t)$, so
$\br e^{(p)}_i (L)(t)\geq \br h(t)$. This readily implies that
$$
 \br \Delta^{(p)}_{k} (t)\geq \sum_{i\in I'(1,p,h(t),L)} \br e^{(p)}_i (L)(t)\geq |I' (1,p,h(t),L)|\cdot \br h(t)\ ,
$$
whence the conclusion.
\end{proof}

Let us now point out the following:

\begin{lem}\label{stimaopposta}
 If $L$ is any link, then
$$
\Aa_\Qq (L)=0\quad \Longrightarrow \quad \Aa (L)=0\ .
$$
\end{lem}
\begin{proof}
Recall from Corollary~\ref{princ:cor}
that 
$\Delta_1^{(p)}(L)(t)=\pi_p
(\Delta (L)(t))$. As a consequence,
if $\Aa(L)>0$, then
$\br \Delta^{(p)} (L)(t)> 0$ for some odd prime $p$ (just choose $p$ to be larger
than the absolute value of all the coefficients of $\Delta(L)(t)$). This implies that $\br e^{(p)}_{i_0} (t)>0$
for some $i_0\in \N$. If $h(t)$ is any irreducible factor of $e^{(p)}_{i_0} (t)$
in $\Lambda_p$, then $i_0\in I(1,p,h(t),L)$, so $\Aa_\Qq (L)\geq I(1,p,h(t),L)=1$.
\end{proof}

The following Corollary readily implies Theorem~\ref{K=K}.

\begin{cor}\label{knot:ino}
 If $L=K$ is a knot, then
$$
I(1,p,h(t),K)\leq \frac{\Aa (K)}{\br h(t)}\leq \Aa(K) ,
$$
$$
\Aa_\Qq(K)=I(K)=I(1,K)\leq \Aa(K)\ ,
$$
$$
\Aa_\Qq(K)=0\quad \Longleftrightarrow\quad \Aa(K)=0\ .
$$
\end{cor}
\begin{proof}
By Corollary~\ref{princ:cor}--(3) we have $\br \Delta_1^{(p)}(K)(t)\leq
\br \Delta (K)(t)=\Aa(K)$, so
 the first inequality follows immediately from Proposition~\ref{Alink}. 
As a consequence, we have $I(1,K)\leq \Aa(K)$, so the second inequality is a consequence
of Lemma~\ref{z=1}. The fact that $\Aa_\Qq(K)=0$ if and only if $\Aa(K)=0$
easily follows from the second inequality and
Lemma~\ref{stimaopposta}.
\end{proof}

\subsection{Computing $\Aa_\Qq$ via proper subfamilies of 
$\Qq_\Ff$}
\label{subfamily}
This Subsection is devoted to determine proper subfamilies of $\Qq_\Ff$ that carry the
whole information about the invariant $\Aa_\Qq$. We will
be mainly interested in the case when
$L=K$ is a knot (some considerations
below hold more generally for $(L,\Pp_m)$). 




Let $K$ be a knot, and recall that
$\delta(K)$  has been defined in Subsection~\ref{performance2}. We begin with the following:

\begin{lem}\label{log}
 Let $f(t)\in\Z[t]$ be a polynomial and suppose that
there exist prime numbers $p_1,\ldots,p_k$ such that 
$$
f(n)=\pm p_1^{\alpha_1(n)}\cdot\ldots\cdot p_k^{\alpha_k(n)}\quad {\rm for\ every}\ n\geq n_0,
$$
where $n_0\in\N$ is fixed. Then $f(t)$ is constant. 
\end{lem}
\begin{proof}
Let $d=\deg f(t)$, and take $h>0$ such that $|f(n)|\leq hn^d$. Then
$\alpha_1(n)\ln p_1+\ldots+\alpha_k(n)\ln p_k\leq  d\ln n +\ln h$ for every $n\geq n_0$. 
In particular,
there exist a constant $w\geq 1$ such that $\alpha_i(n)\leq w\ln n$ for every $n\geq n_0$.
Therefore, if $n_1$ is such that $(n_1-n_0)>2(d+1)(w\ln n_1)^k$, then the interval
$[n_0,n_1]$ contains 
(at least) $2(d+1)$ integers $m_1,\ldots,m_{2(d+1)}$ such that
$\alpha_i(m_j)=\alpha_i (m_{j'})$ for every $i=1,\ldots,k$,
$j,j'=1,\ldots,2(d+1)$, whence  $f(m_i)=\pm f(m_{j'})$ for every $j,j'=1,\ldots,2(d+1)$.
It follows that $f$ takes the same value on at least $d+1$ distinct
integers.
Since
$\deg f=d$, this implies in turn that $f$ is constant. 
\end{proof}

We now prove Proposition \ref{min}, which we recall here for the convenience of the reader.

\begin{prop}\label{subfam}
Let $K$ be a knot.
\begin{enumerate}
\item If $\theta (K)>1$, then 
$\theta(K)\geq \delta(K)+1$.
\item If  $\Aa_\Qq(K)=1$, then $ \delta(K)= 1$.
\item If $\Aa(K)>0$, then $$\delta(K)\leq \frac{\Aa(K)}{\max \{2,\Aa_\Qq(K)\} }\ .$$
\item
Suppose that $\Aa_\Qq(K)=\Aa(K)$ or $\Aa_\Qq(K)=\Aa(K)-1$. Then $\delta(K)=1$.
Moreover, there
  exist an odd prime $p$ and an element
$a\in\Z_p^\ast$ such that $(t-a)^{\Aa_\Qq(K)}$ divides $\Delta_1^{(p)}(K)(t)$ in
  $\Lambda_p$.
\item 
If $\Aa_\Qq(K)=\Aa(K)$, then $\delta(K)=1$ and
there exist an odd prime $p$ and an element $a\in\Z_p^\ast$ such that
$\Delta_1^{(p)}(K)(t)\doteq (t-a)^{\Aa(K)}$ in $\Lambda_p$.
\end{enumerate}
\end{prop}
\begin{proof}  
(1) Suppose that $\Aa_\Qq(K)=|I(1,p,h(t),K)|$, where $\br
h(t)=\delta(K)$. Since $\theta (K)>1$, the quandle $\X=(\F(p,h(t)),*)$ cannot be trivial, so 
Lemma~\ref{type} implies that
$t_\X \geq \delta(K)+1$.

(2) By Theorem~\ref{K=K} we have $\Aa(K)>0$, whence $\br \Delta(K)(t)>0$. 
Let $f(t)\in\Z[t]\subseteq \Lambda$ be such that $f(t)\doteq \Delta(K)(t)$ and
$f(0)\neq 0$, so that $\deg f(t)=\br \Delta(K)(t)=d>0$.
By Lemma~\ref{log}, there exists $n> |f(0)|$ such that $f(n)$  
is divided by a prime number $p> |f(0)|$. Let $a$ be the class
of $n$ in $\Z_p$, and let us set $h(t)=t-a\in\Z_p[t]\subset \Lambda_p$.
Since $p$ divides $f(n)$, we have that $h(t)$ divides $\Delta^{(p)}_1(L)(t)=\pi_p (\Delta(L)(t))$
in $\Lambda_p$. Also observe that $p$ does not divide
$f(0)$, so $p$ does not divide $f(n)-f(0)$, and this readily implies
that $a\neq 0$ in $\Z_p$.  It follows that $h(t)$ is irreducible of positive
breadth in $\Lambda_p$, so we may set $\X=(\F(p,h(t)),\ast)$.
By construction we have $a_\X(K,1)\geq 1=\Aa_\Qq(K)$, so
$\delta(K)=1$.

(3) It is well--known that $\Aa (K)=\br \Delta(K)(t)$ is even, so $\Aa(K)>0$ 
implies that $\Aa(K)\geq 2$. Together with
the inequality $\Aa_\Qq(K)\leq \Aa(K)$, this implies that
$\Aa(K)/\max\{2,\Aa_\Qq(K)\}\geq 1$, so 
we may suppose $\delta(K)>1$, whence $\Aa_\Qq (K)\geq 2$ (see point~(2)).

Suppose now that $\Aa_\Qq (K)=I(1,p,h(t),K)$, where $\delta(K)=\br h(t)$.
Corollary~\ref{knot:ino} implies that $\Aa_\Qq (K)=I(1,p,h(t),K)\leq \Aa(K)/\delta(K)$,
whence the conclusion.

(4) The case $\Aa(K)=0$ is trivial, so the first statement is an immediate consequence of (3). Then, we may choose
$h(t)=(t-a)\in\Lambda_p$, $a\in\Z_p^\ast$, in such a way that
$\Aa_\Qq (K)=|I(1,p,h(t),K)|$. Now
$h(t)$ divides $e^{(p)}_i (K)(t)$ for every $i\in I(1,p,h(t),K)$, so
$h(t)^{\Aa_\Qq(K)}$ divides $\prod_{i\in I(1,p,h(t),K)} e^{(p)}_i (K)(t)$, which
divides in turn $\Delta_1^{(p)}(K)(t)$.

(5) is an immediate consequence of (4).
\end{proof}

\begin{remark}\label{delta:unbounded}
  {\rm As mentioned in Question~\ref{delta:quest}, we are not able to
    prove that $\delta (K)$ may be arbitrarily large. Let us point out
    some difficulties that one has to face in order to prove (or
    disprove) such a statement. If one tries to construct a knot $K$
    with $\delta(K)\geq n$, one has to find $K$ and
    $\X=(\F(p,h(t),\ast)$ such that
    $\Aa_\Qq (K)=a_\X(K,1)=|I(1,p,h(t),K)|$ and $\br h(t)=n$.  
    Once this has been established, the inequality $\delta (K)\leq n$
    is proved.  In order to show that $\delta(K)=n$ we are left to
    prove that for every odd prime $q$ the number of polynomials
    $e^{(q)}_i (K)(t)$, $i\in\N$, admitting a common factor of breadth
    at most $n-1$ is strictly less than $\Aa_\Qq (K)$.  One may
    probably start with a knot $K$ whose Alexander polynomial $\Delta
    (K)(t)$ decomposes as the product of irreducible factors of large
    breadth.  This would ensure that also the $e_i(K)(t)$'s have large
    breadth. However, it is not clear how to control the breadth (and
    the existence of common divisors) of the $e_i^{(q)}(K)(t)$'s, when
    $q$ is a generic prime, even under the hypothesis that
    $I(1,p,h(t),K)$ collects a maximal subset of indices such that the
    corresponding $e_i^{(p)}(K)(t)$'s have a non--trivial common
    divisor (and such a divisor has breadth $n$).
    }
\end{remark}

\section{Genus--1 knots}\label{g=1}  
In this Section we fully describe the case of knots that admit a
Seifert surface of genus 1 (that is, knots of genus 1 or the unknot).
\smallskip

Let us fix a quandle $\X=(\F(p,h(t)),*)$, take $z\in\N$ and consider a
special diagram $(\Dd(T),z)$ of $(K,z)$ (see Section~\ref{ribbon}).




By Lemma~\ref{Qlemma}, the integer $a_\X(K,z)$ is equal to the dimension
of $\ker N(z,p,h(t))$, where 
$$
N(z,p,h(t))=\left(\begin{array}{cc}
  (\overline{t}^z-1)M_{1,1} \qquad & \overline{t}^z(M_{1,2}-1)-M_{1,2}              \\    
   \overline{t}^z(M_{2,1}+1) -M_{2,1}\qquad & (\overline{t}^z-1)M_{2,2}               \end{array}
\right) \ 
$$
(recall that $\overline{t}$ denotes the class of $t$ in $\F(p,h(t))$).

Recall that the \emph{determinant} $\det K$ of $K$ is defined as $\det
K=|\Delta(K)(-1)|=|\det(S(L)+S(L)^T)|$, where $S(L)$ is a Seifert
matrix for $K$.

\begin{lem}\label{AB}
 We have
$$
M_{1,2}-M_{2,1}=1,\qquad \det K=|4\det M -1|\ .
$$
\end{lem}
 \begin{proof}
   The first equality is an immediate consequence of
   Lemma~\ref{B==1}. Putting together Lemmas~\ref{seifert}
   and~\ref{B==1} we also obtain $S(L)+S(L)^T=2M+J$, whence the
   conclusion.
 \end{proof}

\begin{prop}\label{sharp-2g} Let $K$ be a knot such that $g(K)=1$, let 
  us take $z\in\N$ and a quandle
  $\X=(\F(p,h(t)),*)\in \Qq_\Ff$.  Then $$a_\X(K,z)=2$$ if and only if
  for (one, and hence for) every special diagram $(\Dd(T),z)$ of
  $(K,z)$ the following conditions hold:
$$ 
M_{1,1}=M_{2,2}=0,\quad
M_{1,2}=\frac{p+1}{2}, \quad M_{2,1}=\frac{p-1}{2} \qquad {\rm in}\ \Z_p
$$
and
$$  
h(t)\, \big|\,  \left(1+t^z\right) \quad
{\rm in}\ \Lambda_p  \ .
$$
\end{prop}
\begin{proof}
  If $\X$ is trivial, then for every $z$ we have $a_\X
  (K,z)=0$. Moreover, $h(t)=t-1$ does not divide $1+t^z$, so we may
  suppose that $\X$ is non--trivial. Suppose now that $z=0$ in
  $\Z_{t_\X}$. Then we have $a_\X (K,z)=0$, and $h(t)$ has to divide
  $1-t^z$ by Lemma~\ref{type}--(2). As a consequence, $h(t)$ cannot
  divide $1+t^z$, so the conclusion holds also in this case. We may
  therefore assume that $z\neq 0$ in $\Z_{t_\X}$.

  Observe that $a_\X(K,z)=2$ if and only if $N(z,p,h(t))=0$,
  \emph{i.e.}~if and only if
$$ 
\begin{array}{cc}
(\overline{t}^z-1)M_{1,1}=0 \qquad & (\overline{t}^z-1)M_{2,2}=0,\\ 
\overline{t}^z(M_{1,2}-1)-M_{1,2}=0, \qquad &                
   \overline{t}^z(M_{2,1}+1) -M_{2,1}=0\ .
\end{array}
$$
Since $z\neq 0$ in $\Z_{t_\X}$, we have $\overline{t}^z-1\neq 0$, so
the first and the second relations give $M_{1,1}=M_{2,2}=0$.  By
Lemma~\ref{AB}, the third equation can be rewritten as $M_{2,1}
\overline{t}^z=M_{2,1}+1$. Together with the fourth equation, this
immediately implies that $2M_{2,1}=-1$, whence $M_{2,1}=(p-1)/2$, and
$M_{1,2}=(p+1)/2$ again by Lemma~\ref{AB}. Under these conditions, the
third and the fourth equations are equivalent to the fact
$\overline{t}^z+1=0$, \emph{i.e.}~to the fact that $h(t)$ divides
$t^z+1$ in $\Lambda_p$.
\end{proof}

Proposition~\ref{sharp-2g} readily implies the following:

\begin{cor}\label{sharp-2g-2} 
  With notations as in Proposition~\ref{sharp-2g}, we have
$$\Aa_\Qq(K)=2$$ if
and only if there exist a special diagram $\Dd(T)$ of $K$ and a prime
number $p\geq 3$ such that the following equalities hold in $\Z_p$:
$$ 
M_{1,1}=M_{2,2}=0,\quad
M_{1,2}=\frac{p+1}{2}, \qquad M_{2,1}=\frac{p-1}{2}\ .
$$
Moreover, in this case there exists a dihedral quandle $\X\in
\Qq_\Ff(1)$ such that $a_\X(K,1)=2$, so that
$$ 
\delta (K)=1,\qquad \theta(K)=2\ .
$$ 
\end{cor}

\begin{remark}\label{easy} {\rm Notice that if a special diagram
    $(\Dd(T),z)$ verifies the conditions of Corollary~\ref{sharp-2g-2}
	that involve $M_{1,2}$
    and $M_{2,1}$, then we can easily realize also the 
    conditions on $M_{1,1}$ and $M_{2,2}$
via suitable Reidemeister moves of the first
    type (\emph{i.e.}~by ``adding kinks'') on the two strings of $T$.}
\end{remark}

Let us now
discuss the conditions under which $a_\X (K,z)=0$ for every
$\X\in\Qq_\Ff$, $z\in\mathbb{Z}$. 
We begin with the following:

\begin{lem}\label{g1alex}
 We have
$$
\Delta(K)(t)\doteq \det M +(1-2\det M)t +(\det M) t^{2}\ .
$$
\end{lem}
\begin{proof}
Corollary~\ref{seifertcor} implies that 
$$\Delta(K)(t)=\det \left((t-1)M+tJ\right)=(t-1)^2M_{1,1}M_{2,2}-((t-1)M_{1,2}-t)
((t-1)M_{2,1}+t)\ .$$
Since $M_{1,2}-M_{2,1}=1$, the conclusion follows.
\end{proof}


\begin{cor}\label{W(K)}
 The integer
$$
W(K)=\det M
$$
is a well--defined invariant of $K$, \emph{i.e.}~it does not depend
on the special diagram of $K$ defining $M$. Moreover,
$\Delta(K)(t)\doteq \Delta(K')(t)$ if and only if $W(K)=W(K')$.
\end{cor}
\begin{proof}
Suppose that $\Dd'(T)$ is a special diagram of $K$ of genus 1, and let
$M'$ be the matrix encoding the linking numbers of the strings
of $T$. By Lemma~\ref{g1alex} we have 
$$
\det M +(1-2\det M)t +(\det M) t^{2}
\doteq 
\det M' +(1-2\det M')t +(\det M') t^{2},
$$
so $\det M=\pm \det M'$, $
1-2\det M=\pm(1-2\det M')$, and $\det M=\det M'$.
\end{proof}

Putting together Lemma~\ref{g1alex} and Lemma~\ref{a>1} we readily get
the following:

\begin{prop}\label{genus1a1}
Let $K$ be a knot such that $g(K)=1$, let us take
$z\in\N$ and a quandle $\X=(\F(p,h(t)),*)\in \Qq_\Ff$.
Then $$a_\X(K,z)\geq 1$$ if
  and only if 
$$h(t) \; \Big|\; W(K) +(1-2W(K))t^z +W(K) t^{2z} \quad {\rm in}\ 
\Lambda_p\ .$$
In particular, if $W(K)=0$ then $a_\X (K,z)=0$.
\end{prop}

\begin{cor}\label{g=1completo} 
  Let $K$ be a knot, such that $g(K)\leq 1$.  Then $\Aa_\Qq(K)=0$ if
  and only if $W(K)=0$
(due to Lemma~\ref{AB}, this condition is
  equivalent to $\det K=1$). In all the other cases there exists a
  dihedral quandle $\X\in \Qq_\Ff(1)$ such that
  $\Aa_\Qq(K)=a_\X(K,1)\geq 1$. We have in particular
$$ 
\delta (K)=1,\qquad \theta(K)=2\ .
$$ 
\end{cor}
\Dim By Proposition~\ref{genus1a1}, it is sufficient to show that,
if $W(K)\neq 0$, then there exists
 a dihedral quandle $\X\in \Qq_\Ff(1)$ such that $\Aa_\Qq(K)=a_\X(K,1)\geq 1$.

 In fact, if $W(K)\neq 0$ we may choose an odd prime $p$ be dividing
 $1-4W(K)$. Then the polynomial $1+t$ divides $W(K)t^2+(1-2W(K))t+W(K)$ in $\Lambda_p$. By Proposition~\ref{genus1a1}, this
 implies that $a_\X (K,1)\geq 1$, where $\X$ is the dihedral quandle
 $\X=(F(p,1+t),\ast)$.  \cvd

\smallskip 

\begin{remark} {\rm By Corollary~\ref{g=1completo}, every genus--1
    knot such that $\Aa_\Qq(K)\geq 1$ is such that
    $\delta(K)=1$. Since for every such knot we obviously have
    $\Aa(K)\leq 2g(K)=2$, this fact is also a consequence of
    Proposition~\ref{min}.  }
\end{remark}

\subsection {A few manipulations on special diagrams} 
Here below we describe a few simple manipulations on (genus--1) special
diagrams, which are useful to construct large families of examples.

\begin{lem}\label{sharp-3} 
  Let $K$ be a genus--1 knot with  $a_\X(K,z)=2$
for some quandle $\X\in\Qq_\Ff$, 
and let
  $(\Dd(T),z)$ be a special diagram of $(K,z)$.
Then by adding kinks to only one of the two strings of
  $T$ we can arbitrarily modify either $M_{1,1}$ or $M_{2,2}$, so
  that the resulting $(D(T'),z)$ is a special diagram of some $(K',z)$
  such that $a_\X(K',z)=1$.
\end{lem}

\begin{lem}\label{sharp-4}
  Let $K$ be a genus--1 knot with $a_\X(K,z)\geq 1$ for some quandle
  $\X=(\F(p,h(t)),*)$, and let $(\Dd(T),z)$ be a special
  diagram of $(K,z)$.  Let us modify $T$ by means of any sequence of
  usual second and third Reidemeister moves, of first Reidemeister
  moves provided that $M_{1,1}$ and $M_{2,2}$ are kept constant mod $(p)$,
  and of positive (negative) {\rm linking moves} between the two
  strings, (see Figure \ref{linking-move}; here the actual sign of the
  move depends also on the omitted orientations of the strings),
  provided that their number is equal to $0$ mod $\ (p)$. Then we get
  a special diagram $(\Dd(T'),z)$ of some $(K',z)$ such that
  $a_\X(K,z)= a_\X(K',z)\geq 1$.

\end{lem} 
\begin{figure}[htbp]
\begin{center}
 \includegraphics[height=3cm]{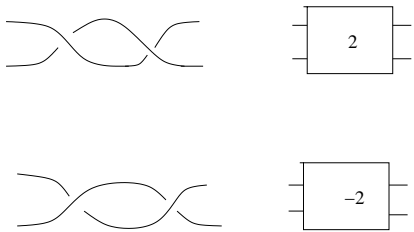}
\caption{\label{linking-move} Linking moves.}
\end{center}
\end{figure}

\medskip 




\subsection{The dihedral case}\label{dihedral} 
Let us specialize the results above to the simplest case of dihedral
quandles, \emph{i.e.} to the case when $\X=\D_p=(\F(p,1+t),\ast)$ and
$z=1$.  In such a case we simply write $$a_p(K)=a_{\D_p}(K,1) \ . $$
The following result is an immediate consequence of
Propositions~\ref{sharp-2g} and~\ref{genus1a1}.

\begin{lem}\label{dihedral:lem} Let $K$ be a knot such that $g(K)=1$,
represented by a special diagram $\Dd(T)$.
\begin{enumerate}
\item $a_p(K)=2$ if and only the following system of relations is
  satisfied in $\Z_p$:
$$M_{1,2}=\frac{p+1}{2}, \ M_{2,1} = \frac{p-1}{2}, \ M_{2,2}=0, \  M_{1,1}=0 \ . $$
\item $a_p(K)\geq 1$ if and only if 
$$1-4W(K) =0\quad {\rm in}\ \Z_p\ .$$
By Lemma~\ref{AB}, this condition holds if and only if
$p$ divides $\det K$.
\end{enumerate}
\end{lem}

\smallskip

Now we want to show that, for every $p>2$, the first set of conditions
in the last Lemma can be actually realized by a special diagram
$\Dd(T_p)$ of some knot $K_p$. Let us consider the tangle of Figure
\ref{T(p)}. Here $p=2k+1$ and there are $k$ (resp.~$k+1$) overcrossing
(resp.~undercrossing) vertical strands. So it is immediate to verify
that (in $\Z$):
$$ M_{1,1}=0, \quad M_{2,2}=p,\quad M_{1,2}=- \frac{p-1}{2}, \quad M_{2,1}
=-\frac{p+1}{2} \ , $$
whence
$$
1-4W(K)=1-4\det M=p^2\ .
$$ 
\begin{figure}[htbp]
\begin{center}
 \includegraphics[height=6cm]{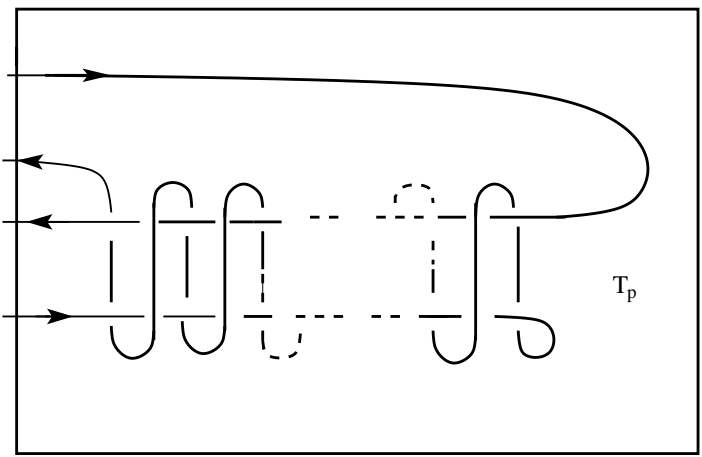}
\caption{\label{T(p)} The tangle $T(p)$.}
\end{center}
\end{figure}
Hence:

\begin{lem}\label{Kp} The family $\{K_p\}$ of genus--1 
  knots constructed above is such that $a_p(K_p)=2$ while
  $a_{p'}(K_p)=0$ for every $p'\neq p$. In particular, $K_p$ is not
  isotopic to $K_{p'}$ if $p\neq p'$, and $\Aa_\Qq(K_p)=2$ for every
  $p$.
\end{lem}

\medskip

Let us now modify $T_p$ into a tangle $T'_p$ by adding one positive
kink to the second string of $T$, and let us denote by $K'_p$ the knot
described by the special diagram $\Dd (T'_p)$.

\begin{lem}\label{K'p}  
  The family $\{K'_p\}$ of genus--1 knots just constructed is such that
  $a_p(K'_p)=1$ while $a_{p'}(K'_p)=0$ for every $p'\neq p$. In
  particular, $K'_p$ is not isotopic to $K'_{p'}$ if $p\neq p'$, and
  $\Aa_\Qq(K'_p)=1$ for every $p$. Moreover, $\Delta(K_p)(t)\doteq \Delta(K'_p)(t)$
for every odd prime $p$.
\end{lem}
\begin{proof}
  It is readily seen that $W(K'_p)=W(K_p)=(p^2-1)/4$, so
$\Delta(K'_p)(t)=\Delta(K_p)(t)$ and 
  $a_p(K'_p)\geq 1$, while $a_{p'}(K'_p)=0$ for every $p'\neq
  p$. Moreover, $a_p (K'_p)\neq 2$ since for the tangle $T'_p$ the
  value of $M_{2,2}$ is equal to $p+1$ which is not null in $\Z_p$.
\end{proof}

\begin{figure}[htbp]
\begin{center}
 \includegraphics[height=5cm]{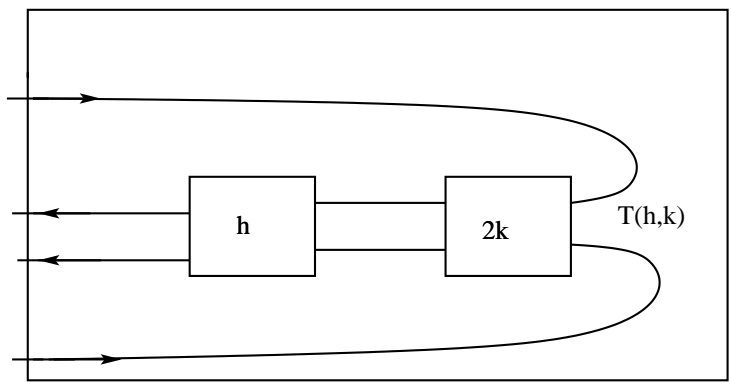}
\caption{\label{h-k} 
The tangle $T(h,k)$. The rectangular boxes refer to the tangles
described in Figure~\ref{linking-move}.}
\end{center}
\end{figure}

Lemmas~\ref{Kp} and~\ref{K'p} imply that the quandle invariant $\Aa_\Qq$ can be more
effective than
the Alexander polynomial in distinguishing knots, and
this phenomenon shows up already in the case of genus--1 knots:

\begin{cor}\label{diversi:cor}
 There exist genus--1 knots $K$, $K'$ such that $\Delta(K)(t)\doteq \Delta(K')(t)$,
while $\Aa_\Qq (K)=2$ and $\Aa_\Qq(K')=1$. Moreover, 
$\Aa_\Qq(K)$ and $\Aa_\Qq(K')$ may be realized by the same dihedral quandle and the same
cocycle $z=1$. 
\end{cor}





\subsection{An example involving a quandle of order $p^2$}
All the previous explicit examples are obtained by
using some Alexander quandle structure on $\F_p$. By
Corollary~\ref{g=1completo}, such quandle structures
encode the relevant information about the invariant $\Aa_\Qq$
of genus--1 knots.

Let us show anyway also an example based on a quandle of order $p^2$.
Consider the tangle $T(h,k)$ of Figure \ref{h-k},
encoding a knot $K(h,k)$. One can verify that $a_\X(K(5,3),2)=2$, when
$\X=(\F(11,1+t^2),*)$, which is of type $t_\X>2$ (here $q=11^2$).




\subsection{The general picture of genus--1 knots}
The following Proposition summarizes the discussion carried out in the preceding Subsections:

\begin{prop}\label{AFvAg=1} Let $K$ and $K'$ be knots of genus $g\leq 1$.
Then:
\begin{enumerate}
\item
Let $M$ be the matrix associated to a special diagram of $K$. Then
the integer $W(K)=\det M$ is a well--defined invariant of $K$
(\emph{i.e.}~it does not depend on the chosen diagram).

\item  $\Delta(K)(t)\doteq \Delta(K')(t)$ if and only if $W(K)=W(K')$.

\item $\Aa(K)=0$ if and only if $\Aa_\Qq(K)=0$ if and only if $W(K)=0$.

\item $\Aa(K) \in \{0,2\}$, while $\Aa_\Qq(K)\in \{0,1,2\}$.
More precisely, for every $\eta\in \{0,1,2\}$ there exists a genus--1 knot
$K$ such that $\Aa_\Qq(K)=\eta$.

\item There exist $K$ and $K'$ such that  $\Delta(K)(t)=\Delta(K')(t)$,
while $\Aa_\Qq(K)=2$ and $\Aa_\Qq(K')=1$.
\end{enumerate}
\end{prop}
\begin{proof}
By Lemmas~\ref{g1alex}, Corollaries~\ref{W(K)}, \ref{g=1completo}, \ref{diversi:cor},
we are only left to prove that 
there exists a genus--1 knot $K$ such that $\Aa_\Qq(K)=0$. 
As an example of such a knot, one may take any genus--1 knot with trivial Alexander polynomial, such as
the Whitehead double of the figure--eight knot.
\end{proof}

\subsection{On genus--1 knots with minimal Seifert rank}
    Recall that a Seifert surface $\Sigma$ of a knot $K$ is said to have
    {\it minimal Seifert rank} if the rank of its Seifert form $S$ equals
    the genus $g=g(\Sigma)$. Moreover, a knot has {\it minimal Seifert rank} if it admits
    a Seifert surface (of arbitrary genus) having minimal
    Seifert rank. It is a well--known fact that every knot $K$ with minimal Seifert
    rank has trivial Alexander polynomial $\Delta(K)(t)\doteq 1$
    (i.e. $\Aa(K)=0$). We claim that:

\smallskip

{\it If $K$ is a knot such that $g(K)\leq 1$ and $\Aa(K)=0$, then
every genus--$1$ Seifert surface for $K$ has minimal Seifert rank.
It follows that a genus--$1$ knot has trivial Alexander polynomial if and only
if it has minimal Seifert rank.}

\smallskip

In fact, let $\Dd$ be a special diagram for $K$ associated to a given
genus--$1$ Seifert surface $\Sigma$, let $M$ be the matrix associated to $\Dd$, and
let $S$ be the matrix representing the Seifert form on $\Sigma$ 
with respect to the
geometric basis carried by $\Dd$. Corollary~\ref{seifertcor} implies
that
$$
S=\left( \begin{array}{cc} M_{1,1} & M_{1,2}-1\\
          M_{2,1}+1 & M_{2,2}
         \end{array}\right)\ ,
$$
so $S$ has rank equal to $1$ if and only if $0=\det S=\det M
-M_{1,2}+M_{2,1}+1=\det M=W(K)$.  By Proposition~\ref{AFvAg=1}--(3),
if $\Aa(K)=0$ then $W(K)=0$, so $S$ has minimal rank.

It is a non--trivial fact proved in \cite{GT} that the last statement
of the claim does not hold in general for knots of genus $\geq 2$.

\section {Sums of genus--1 knots}
\label{g-arbitrario}
We can use genus--1 knots as buiding blocks for the
construction of examples of arbitrary genus. Let us first observe
that, if $K$ and $K'$ are (oriented) knots endowed respectively with
special diagrams $\Dd(T)$ and $\Dd(T')$, then the knot $K+K'$ admits
an obvious special diagram $\Dd(T+T')$ (see Figures~\ref{sum:fig}
and~\ref{sum2:fig}).

\begin{figure}[ht]
\begin{center}
 \includegraphics[width=8cm]{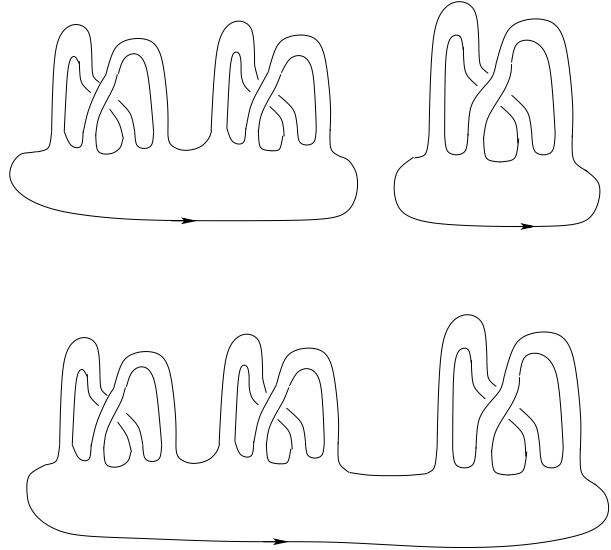}
\caption{\label{sum:fig} On the top: two special diagrams $\Dd,\Dd'$ of the unknot
$K_0$. On the bottom:
the special diagram for $K_0+K_0=K_0$ obtained by ``summing'' the special diagrams
on the top.}
\end{center}
\end{figure}

\begin{figure}[ht]
\begin{center}
 \includegraphics[width=8cm]{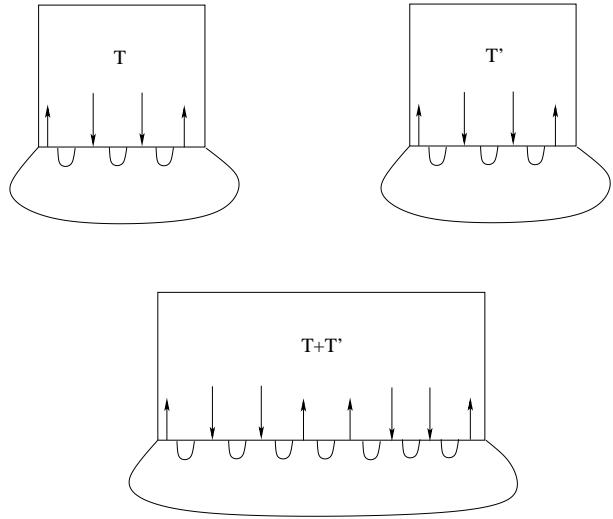}
\caption{\label{sum2:fig} If $T$ and $T'$ are the primary tangles of special
diagrams of $K$ and $K'$, then $T+T'$ is the primary tangle of
a special diagram of $K+K'$.}
\end{center}
\end{figure}



Let $N(z)$, $N'(z)$, $N''(z)$ be the matrices associated to the special diagrams
$\Dd(T)$, $\Dd(T')$, $\Dd(T+T')$ 
as in Section~\ref{main-dim}, where $z$ is a natural number. It is immediate to realize that 
$$
N''(z)=\left( \begin{array}{ccc} N(z) & \vline & 0\\
              \hline 
	0 & \vline & N'(z)
              \end{array}\right)\ .
$$
Since $N(1)$ (resp.~$N'(1)$, $N''(1)$) is 
a Seifert matrix for $K$ (resp.~$K'$, $K+K'$), this readily implies the well--known:

\begin{lem}\label{additive2} 
We have
$$ \Delta(K_1 + K_2 + \dots + K_h)(t)= \prod_{j=1}^h \Delta(K_j)(t)\ ,$$
whence
$$ \Aa(K_1 + K_2 + \dots + K_h)= \sum_{j=1}^h \Aa(K_j) \ . $$ 
\end{lem}

Let us now fix an odd prime $p$ and an irreducible element
$h(t)\in\Lambda_p$ of positive breadth.
Since $N(z,p,h(t))$
(resp.~$N'(z,p,h(t))$, $N''(z,p,h(t))$) is obtained
from $N(z)$ (resp.~$N'(z)$, $N''(z)$) just by projecting the coefficients
onto $\F(p,h(t))$, from Lemma~\ref{Qlemma} we deduce the following: 

\begin{lem}\label{additive} 
For every $\X\in \Qq_\Ff$, $z\in \Z_{t_\X}$,  we have
$$
a_\X (K_1+K_2+\ldots+K_h,z)=\sum_{i=1}^h a_\X (K_j,z)\ .
$$
Therefore,
$$ \Aa_\Qq(K_1 + K_2 + \dots + K_h,z)\leq \sum_{j=1}^h \Aa_\Qq(K_j) \ . $$
\end{lem}

We observe that the equality $\Aa_\Qq(K_1 + K_2 + \dots + K_h,z)= \sum_{j=1}^h \Aa_\Qq(K_j)$
does not hold in general. The equality holds if a single quandle
$\X\in\Qq_\Ff$ exists which realizes all the $\Aa_\Qq(K_i)$'s with respect
to the same cycle.

\smallskip

We are now ready to prove Proposition~\ref{betterinvariant},
which we recall here for the convenience of the reader.

\begin{prop}
Let us fix $g\geq 1$. Then, for every $r_1,r_2$ such that
$1\leq r_1\leq r_2 \leq 2r_1 \leq 2g$, there exist
  knots $K_1$ and $K_2$ such that the following conditions hold: 
$$g(K_1)=g(K_2)=g, \qquad  \Delta (K_1)=\Delta (K_2)
\ {\rm (whence}\  \Aa(K_1)=\Aa(K_2){\rm )}\ ,
$$
while $$\Aa_\Qq(K_1)=r_1, \qquad \Aa_\Qq(K_2)= r_2\ .$$ 
Moreover, we
  can require that both $\Aa_\Qq(K_1)$ and $\Aa_\Qq(K_2)$ are realized by
  means of some dihedral quandle with cycle $\overline{z}=1$. 
\end{prop}
\begin{proof}
 Let $K,K'$ be the genus--1 knots provided by Corollary~\ref{diversi:cor},
and let $K''$ be a genus--1 knot with trivial Alexander polynomial
(see Proposition~\ref{AFvAg=1}. Then we may define $K_1$ as the sum
of $r_1$ copies of $K'$ and $g-r_1$ copies of $K''$, and 
$K_2$ as the sum of $2r_1-r_2$ copies of $K'$, $r_2-r_1$
copies of $K$ and $g-r_1$ copies of $K''$. The additivity of the genus
gives that $g(K_1)=g(K_2)=g$, and
Lemma~\ref{additive2} readily implies that $\Delta(K_1)(t)=\Delta(K_2)(t)$.
Moreover, by Lemma~\ref{additive} we have that 
$\Aa_\Qq(K_1)\leq r_1$ and $\Aa_\Qq(K_2)\leq r_2$. However, 
Corollary~\ref{diversi:cor} ensures that
there exists a dihedral quandle $\X$ such that $\Aa_\Qq(K)=a_\X (K,1)=2$ and $\Aa_\Qq(K')=a_\X(K',1)=1$, so by Lemma~\ref{additive}
$a_\X(K_1,1)=r_1$ and $a_\X (K_2,1)=r_2$, whence the conclusion. 
\end{proof}



\subsection{The case of links}\label{linkadd:sub}
Let $L=K_1\cup\ldots \cup K_h$ be a split link,
where $K_i$ is a knot for every $i=1,\ldots,h$. Let also
$\Pp_M$ be the maximal partition of $L$, let $\overline{z}\colon\Pp_M\to \N$
be a $\Pp_M$--cycle and 
set $z_i=\overline{z}(K_i)$.
The following Lemma is an immediate consequence of the definition
of quandle coloring:

\begin{lem}\label{link:add}
 For every quandle $\X\in\Qq_\Ff$ we have
$$
c_\X (L,\Pp_M,\overline{z})=\prod_{i=1}^h c_\X(K_i,z_i)\ ,
$$
so
$$
a_\X(L,\Pp_M,\overline{z})=\left(\sum_{i=1}^h a_\X (K_i,z_i)\right)+h-1\ ,
$$
and
$$
\Aa_\Qq (L,\Pp_M)\leq \left(\sum_{i=1}^h \Aa_\Qq (K_i)\right)+h-1\ .
$$
Moreover, if $h\geq 2$ then $\Delta(L)(t)=0$, so $\Aa(L)=0$.
\end{lem}

Just as in Lemma~\ref{additive}, the equality
$
\Aa_\Qq (L,\Pp_M)\leq \left(\sum_{i=1}^h \Aa_\Qq (K_i)\right)+h-1
$
does not hold in general.

Let now $K_0$ be a genus--1 knot such that $\Aa_\Qq(K_0)=2$ (see Section~\ref{g=1}
for examples of such knots), and let
$L_h$ be the split link having $h$ components, each isotopic to $K_0$.
The following result implies Proposition~\ref{nobounds}

\begin{prop}
We have 
$\Aa_\Qq(L_h)=3h-1$.
\end{prop}
\begin{proof}
We have $a_\X(K_0,1)=2$ for some quandle $\X\in\Qq_\Ff$, so Lemma~\ref{link:add}
implies that 
$$
\Aa_\Qq(L_h)\geq a_\X(L,\Pp_m,\overline{1})=a_\X(L,\Pp_M,\overline{1})=3h-1\ .
$$
\end{proof}

\begin{remark}
{\rm
Strictly speaking, the equality $\Aa_\Qq(L_h)=3h-1$ does not provide a sharp bound
on $g(L_h)$, since Corollary~\ref{lowerbound} states that
$\Aa_\Qq(L)\leq 2g(L)+2k-2$ for every $k$--component link. This inequality
provides the bound $2g(L_h)\geq h+1$, which is not sharp since
of course $g(L_h)=h$. However, it is immediate to see that if $L$ is a split link,
then $g(L)=g(L,\Pp_M)$. The inequalities
$$
3h-1=\Aa_\Qq(L_h)\leq \Aa_\Qq(L_h,\Pp_M)\leq 2g(L_h,\Pp_M)+h-1=2g(L_h)+h-1
$$
imply now $g(L_h)\geq h$.
} 
\end{remark}

\end{document}